\documentclass[12pt,leqno,a4paper]{amsart}
\usepackage{amssymb,enumerate}
\newcommand{\CC}{{\mathbb{C}}}
\newcommand{\FF}{{\mathbb{F}}}
\newcommand{\ZZ}{{\mathbb{Z}}}

\newcommand{\fA}{{\mathfrak{A}}}
\newcommand{\fS}{{\mathfrak{S}}}

\newcommand{\bG}{{\mathbf{G}}}
\newcommand{\bH}{{\mathbf{H}}}
\newcommand{\bT}{{\mathbf{T}}}

\newcommand{\cE}{{\mathcal{E}}}
\newcommand{\cF}{{\mathcal{F}}}
\newcommand{\cO}{{\mathcal{O}}}

\newcommand{\diag}{{\operatorname{diag}}}
\newcommand{\Ind}{{\operatorname{Ind}}}
\newcommand{\Res}{{\operatorname{Res}}}
\newcommand{\Irr}{{\operatorname{Irr}}}
\newcommand{\St}{{\operatorname{St}}}
\newcommand{\Syl}{{\operatorname{Syl}}}
\newcommand{\sR}{{{{}^*\!R}}}

\newcommand{\Gal}{{\operatorname{Gal}}}
\newcommand{\GL}{{\operatorname{GL}}}
\newcommand{\PGL}{{\operatorname{PGL}}}
\newcommand{\PSL}{{\operatorname{L}}}
\newcommand{\SL}{{\operatorname{SL}}}
\newcommand{\Sp}{{\operatorname{Sp}}}

\newcommand{\PSp}{{\operatorname{S}}}
\newcommand{\OO}{{\operatorname{O}}}

\newcommand{\GU}{{\operatorname{GU}}}
\newcommand{\PGU}{{\operatorname{PGU}}}
\newcommand{\SU}{{\operatorname{SU}}}
\newcommand{\PSU}{{\operatorname{U}}}
\DeclareMathOperator{\lsb}{sb}    
\DeclareMathOperator{\lse}{se}    

\newcommand{\Ph}[1]{\Phi_{#1}}
\newcommand{\tw}[1]{{}^#1\!}
\newcommand{\GAP}{{\sf GAP}}
\newcommand{\Chevie}{{\sf Chevie}}
\newcommand\tvhi{{\tilde\varphi}}

\let\eps=\epsilon
\let\la=\lambda
\let\vhi=\varphi
\let\lra=\longrightarrow

\newtheorem{thm}{Theorem}[section]
\newtheorem{lem}[thm]{Lemma}

\newtheorem{cor}[thm]{Corollary}
\newtheorem{prop}[thm]{Proposition}

\theoremstyle{definition}
\newtheorem{exmp}[thm]{Example}

\theoremstyle{remark}
\newtheorem{rem}[thm]{Remark}

\begin{document}

\title[Simple endotrivial modules]
      {Simple endotrivial modules for linear,\\ unitary and exceptional groups}

\date{\today}

\author{Caroline Lassueur and Gunter Malle }
\address{FB Mathematik, TU Kaiserslautern, Postfach 3049,
         67653 Kaisers\-lautern, Germany.}
\email{lassueur@mathematik.uni-kl.de}
\email{malle@mathematik.uni-kl.de}

\thanks{The authors gratefully acknowledge financial support by ERC
  Advanced Grant 291512. The first author also acknowledges financial support
  by SNF Fellowship for Prospective Researchers PBELP2$_{-}$143516}



\begin{abstract}
Motivated by a recent result of Robinson showing that simple endotrivial
modules essentially come from quasi-simple groups we classify such modules
for finite special linear and unitary groups as well as for exceptional groups
of Lie type. Our main tool is a lifting result for endotrivial modules
obtained in a previous paper which allows us to apply character theoretic
methods. As one application
we prove that the $\ell$-rank of quasi-simple groups possessing a faithful
simple endotrivial module is at most~2. As a second application we complete
the proof that principal blocks of finite simple groups cannot have Loewy
length~4, thus answering a question of Koshitani, K\"ulshammer and Sambale.
Our results also imply a vanishing result for irreducible characters of special
linear and unitary groups.
\end{abstract}

\maketitle

\section{Introduction} \label{sec:intro}

Let $G$ be a finite group and $k$ a field of prime characteristic $\ell$
dividing $|G|$. A $kG$-module $V$ is called \emph{endotrivial} if
$V\otimes V^*\cong k\oplus P$, with a projective $kG$-module $P$. Endotrivial
modules have seen a considerable interest in the last fifteen years. Due to
a recent result of Robinson, the focus has moved to simple endotrivial modules
for quasi-simple groups. In our predecessor paper \cite{LMS} we classified
such modules for several families of finite quasi-simple groups. The present
paper is a continuation of this project. We obtain a complete classification
for special linear and unitary groups and almost complete results for groups
of exceptional Lie type.   \par
We also give a necessary and sufficient condition
for a trivial source module to be endotrivial, depending only on the values
of the associated ordinary character on $\ell$-elements of the group. This
allows us to settle some cases left open in \cite{LMS} for simple modules of
sporadic groups. \par
It turned out in our previous work \cite{LMS} that simple modules of
quasi-simple groups $G$ are rarely ever endotrivial, and if such modules exist
at all, then this seems to severely restrict the structure of Sylow
$\ell$-subgroups of $G$. Combining our new results with those of \cite{LMS}
and \cite{LM15} we can show Conjecture~1.1 in \cite{LMS}:

\begin{thm}   \label{thm:l-rank}
 Let $G$ be a finite quasi-simple group having a faithful simple endotrivial
 module in characteristic~$\ell$. Then the Sylow $\ell$-subgroups of $G$ have
 rank at most~2.
\end{thm}

In the case of cyclic Sylow $\ell$-subgroups, we calculated in \cite[\S3]{LMS} 
the number of $\ell$-blocks of a group containing simple endotrivial modules 
and proved that they correspond to non-exceptional liftable characters.
In $\ell$-rank $2$, simple endotrivial modules only seldom occur.
We show that the following covering groups of linear, unitary and 
exceptional groups do have faithful simple endotrivial modules in
$\ell$-rank~$2$: 
for $\ell=3$, the groups $2.\PSL_3(4)$, $4_1.\PSL_3(4)$, $4_2.\PSL_3(4)$ and 
$\PSU_3(q)$ with $q\equiv2,5\pmod9$,  for $\ell=5$, 
the groups $F_4(2)$, $2.F_4(2)$, and for $\ell=7$, the group $2.F_4(2)$. 
We note that all these modules have trivial source, apart from the ones
for $4_1.\PSL_3(4)$.

As a further application of our results we are able to complete an
investigation begun by Koshitani, K\"ulshammer and Sambale \cite{KKS} and show:

\begin{thm}   \label{thm:LL}
 Let $G$ be a finite non-abelian simple group and $\ell>2$ a prime such that
 Sylow $\ell$-subgroups of $G$ are noncyclic. Then the principal $\ell$-block
 of $G$ cannot have Loewy length~4.
\end{thm}

Our proofs also yield a vanishing result for characters
of linear and unitary groups of large enough $\ell$-rank:

\begin{thm}   \label{thm:rank3}
 Let $\ell>2$ be a prime and $G$ a finite quasi-simple covering group of
 $\PSL_n(q)$ or $\PSU_n(q)$ of $\ell$-rank at least~3. Then for any non-trivial
 $\chi\in\Irr(G)$ there exists an $\ell$-singular element $g\in G$ with
 $\chi(g)=0$, unless one of:
 \begin{enumerate}
  \item[\rm(1)] $\ell|q$ is the defining characteristic of $G$;
  \item[\rm(2)] $\ell=5$, $G=\PSL_5(q)$ with $5||(q-1)$ and $\chi(1)=q^2\Phi_5$;
   or
  \item[\rm(3)] $\ell=5$, $G=\PSU_5(q)$ with $5||(q+1)$ and
   $\chi(1)=q^2\Phi_{10}$.
 \end{enumerate}
\end{thm}

One of the main tools for our investigations is our earlier result
\cite[Thm.~1.3]{LMS} which asserts that all endotrivial modules are liftable to
characteristic~0. This allows us to apply character theoretic methods in our
investigation. More specifically, be can employ the ordinary character theory
of groups of Lie type as developed by Lusztig.
\smallskip

The paper is built up as follows. In Section~\ref{sec:pre} we derive
Theorem~\ref{prop:torchar} and apply it first in Corollary~\ref{cor:casesLMS}
to rule out some examples in sporadic and exceptional type groups, and in
Corollary~\ref{cor:torsexmp} to
compute the torsion subgroup $TT(G)$ of $T(G)$ for several small sporadic
simple groups. In Section~\ref{sec:linear} we show that special linear groups
in rank at least~3 have no faithful simple endotrivial modules (see
Theorem~\ref{thm:SLn}). 
In Section~\ref{sec:unitary} we classify simple endotrivial modules for
special unitary groups (see Theorem~\ref{thm:SUn}) and complete the proof of
Theorem~\ref{thm:l-rank} (in Section~\ref{subsec:SUn}) and of
Theorem~\ref{thm:rank3} (in Section~\ref{subsec:zeroes}). We investigate simple
endotrivial modules for exceptional type groups in Section~\ref{sec:exc},
leaving open only one situation in groups of types $E_6$, $\tw2E_6$ and $E_7$,
respectively, see Theorem~\ref{thm:exc}. Finally, in Section~\ref{sec:QKKS}
we complete the proof of Theorem~\ref{thm:LL}.

\medskip
\noindent
{\bf Acknowledgement:} We thank Frank L\"ubeck for information on the Brauer
trees of $6.\tw2E_6(2)$ for the prime~13, and Shigeo Koshitani for
drawing our attention to \cite{KKS}.

\section{Characters of trivial source endotrivial modules} \label{sec:pre}

Let $G$ be a finite group, $\ell$ a prime number dividing $|G|$,
$k$ an algebraically closed field of characteristic $\ell$. Let $(K,\cO,k)$
be a splitting $\ell$-modular system for $G$ and its subgroups. Let
$1_G\in\Irr(G)$ denote the trivial character of $G$.  \par
If $M$ is an indecomposable trivial source $kG$-module, then $M$ lifts uniquely
to an indecomposable trivial source $\cO G$-lattice, denoted henceforth by
$\hat{M}$. Let $\chi_{\hat M}$ denote the ordinary character of $G$ afforded by
$\hat M$. The following was proven by Landrock--Scott, see
\cite[II, Lem.~12.6]{Lan}:

\begin{lem} \label{lem:tschar}
 Let $M$ be an indecomposable trivial source $kG$-module and $x\in G$ be an
 $\ell$-element. Then:
 \begin{enumerate}
  \item[\rm(a)] $\chi_{\hat M}(x)\geq 0$ is an integer (corresponding to the
   multiplicity of the trivial $k\!\left<x\right>$-module as a direct summand
   of $M|_{\left<x\right>}$);
  \item[\rm(b)] $\chi_{\hat M}(x)\neq 0$ if and only if $x$ belongs to a vertex
   of $M$.
 \end{enumerate}
\end{lem}

From this we derive the following necessary and sufficient condition:

\begin{thm}   \label{prop:torchar}
 Let $M$ be an indecomposable trivial source $kG$-module. Then $M$ is
 endotrivial if and only if $\chi_{\hat M}(x)=1$ for all non-trivial
 $\ell$-elements $x\in G$.
\end{thm}

\begin{proof}
If $\dim_{k}M=1$, then $M$ is endotrivial and satisfies $\chi_{\hat M}(x)=1$
for all $\ell$-elements $x\in G$, therefore we may assume $\dim_{k}M>1$.
We may also assume $\dim_{k}M$ is prime to~$\ell$. Indeed, on the one hand
endotrivial modules have dimension prime to $\ell$ (see \cite[Lem.~2.1]{LMS}),
and on the other hand if the sufficient condition is assumed, then
$\dim_kM\equiv\chi_{\hat M}(x)=1 \pmod\ell$ for any non-trivial $\ell$-element
$x\in G$. Thus, by \cite[Thm.~2.1]{BC86}, the trivial module occurs
as a direct summand of $M\otimes_k M^{*}$ with multiplicity~1. Since $M$ has
trivial source we may write
$$M \otimes_k M^{*}\cong k\oplus N_1\oplus\ldots\oplus N_{r}$$
where $N_1,\dots, N_{r}$ are non-trivial indecomposable trivial source
modules. Clearly $M$ is endotrivial if and only if the module $N_{i}$ has
vertex $\{1\}$ for every  $1\leq i\leq r$. At the level of characters we have
$$\chi_{\hat M}\chi_{\hat M^{*}}
  =1_G+\chi_{\hat{N_1}}+\cdots+\chi_{\hat{N_r}}\,.$$
Thus, by Lemma~\ref{lem:tschar}(b), $M$ is endotrivial if and only if
$\left(\chi_{\hat M}\chi_{\hat M^{*}}\right)(x)=1$ for every non-trivial
$\ell$-element $x\in G$. Moreover since $\dim_kM$ is prime to $\ell$, $M$ has
vertex a Sylow $\ell$-subgroup $P$, so that, by
Lemma~\ref{lem:tschar}(a) and (b), $\chi_{\hat M}(x)$ is a positive integer
for every non-trivial $\ell$-element $x\in G$. The claim follows.
\end{proof}

Henceforth we denote by $T(G)$ the abelian group of isomorphisms classes of
indecomposable endotrivial modules, with multiplication induced by the tensor
product $\otimes_{k}$. Then $T(G)$ is known to be finitely generated. For
notation and background material on $T(G)$ we refer to the survey \cite{C12}
and the references therein.   \par
Under the assumption that the normal $\ell$-rank of the group $G$ is greater
than one, trivial source endotrivial modules coincide with torsion endotrivial
modules. Then Theorem~\ref{prop:torchar} says that the torsion subgroup $TT(G)$
of $T(G)$ is a function of the character table of $G$.

\begin{cor}   \label{cor:selfdual}
 Assume the Sylow $\ell$-subgroups of $G$ are neither cyclic, nor semi-dihedral,
 nor generalised quaternion. If $V$ is a self-dual endotrivial $kG$-module,
 then $\chi_{\hat V}(1)\equiv 1\pmod{|G|_\ell}$ and $\chi_{\hat V}(x)=1$ for
 all non-trivial $\ell$-elements $x\in G$.
\end{cor}

\begin{proof}
If $V$ is self-dual, then the class of $V$ is a torsion element of the group of
endotrivial modules $T(G)$. If $P\in \Syl _\ell(G)$, then restriction to $P$
induces a group homomorphism $\Res^G_{P}:T(G)\lra T(P):[M]\mapsto[M|_{P}]$,
where under the assumptions on $P$ the group  $T(P)$ is torsion-free (see
\cite[\S7 and~\S8]{Th}). It follows that $V|_{P}\cong k\oplus (\text{proj})$,
that is, $V$ is a trivial source $kG$-module. The congruence modulo
$|G|_\ell$ is immediate and the second claim is given by
Theorem~\ref{prop:torchar}.
\end{proof}

This allows us to settle some of the cases left open in \cite{LMS}
corresponding to self-dual modules not satisfying the conclusion of
Corollary~\ref{cor:selfdual}:

\begin{cor}   \label{cor:casesLMS}
 The candidate characters for the sporadic groups $Fi'_{24}$ and $M$ from
 \cite[Table 6]{LMS} are not endotrivial. Similarly neither are the unipotent
 characters $\phi_{80,7}$ for $E_6(q)$ nor $\phi_{16,5}$ for $\tw2E_6(q)$ from
 \cite[Table 4]{LMS}.
\end{cor}

The following examples show how Theorem~\ref{prop:torchar} enables us to
find torsion endotrivial modules using induction and restriction of characters
and to settle some other cases for covering groups of sporadic simple groups
left open in \cite[Table 6]{LMS} or which could so far be discarded via
{\sf Magma} computations only.
Below we denote ordinary irreducible characters by their degrees, and the
labelling of characters and blocks is that of the $\GAP$ character table
library \cite{GAP}.

\begin{exmp}[The character of degree 55 of $M_{22}$ in characteristic $\ell=3$]
\label{exmp:m22} \vbox{\  }
Let $G=M_{22}$ and let $P\in\Syl_3(G)$. Let $V_{55}$ denote the simple
$kG$-module affording the character $55_1\in\Irr(G)$. Let $e_0$ denote the
block idempotent corresponding to the principal $3$-block of $kG$, which is
the unique block with full defect. The group  $G$ has a maximal subgroup
$H\cong\fA_6.2_3$ such that ${H\geq N_G(P)}$, and inducing the non-trivial
linear character $1_2\in\Irr(H)$ to $G$ yields $e_0\cdot \Ind_H^G(1_2)=55_1$,
so that $55_1$ is the character of a trivial source module.
Thus $V_{55}$ is endotrivial by Theorem~\ref{prop:torchar}.
\end{exmp}

\begin{exmp}[The two faithful characters of degree 154 of $2.M_{22}$ in
characteristic $\ell=3$]   \label{exmp:2m22} \vbox{\  }
Let $G=2.M_{22}$, let $P\in\Syl_3(G)$. Let $V_{154},V_{154}^{*}$
denote the dual simple $kG$-modules of dimension 154 affording the characters
$154_2,154_3\in \Irr(G)$. They both belong to the faithful block $B_{6}$.
The group $G$ has a subgroup $H\cong (2\times \fA_6).2_3\geq N_G(P)$, and
inducing the linear characters $1_3,1_4\in\Irr(H)$ to $G$ we have
$e_{6}\cdot \Ind_H^G(1_3)=154_2$ and $e_{6}\cdot \Ind_H^G(1_4)=154_3$,
where $e_{6}$ is the block idempotent corresponding to $B_{6}$. Therefore
$V_{154}$, $V_{154}^*$ are trivial source modules, thus by
Theorem~\ref{prop:torchar} it follows from the values of $154_2,154_3$
that they are endotrivial. We note that the two faithful simple endotrivial
modules $V_{10},V_{10}^*$ affording the characters $10_1,10_2\in\Irr(G)$
from \cite[Thm.~7.1]{LMS} are also trivial source modules. Indeed
$10_1\otimes10_1=45_2+55_1$ where $45_2$ has defect zero.
Whence $V_{10}\otimes_{k} V_{10}\cong V_{55}\oplus\text{(proj)}$ and
$V_{10},V_{10}^*$ must be trivial source since $V_{55}$ is. \par
\end{exmp}

\begin{exmp}[The three characters of degree 154 of $HS$ in characteristic $\ell=3$]
\label{exmp:HS}  \vbox{\  }
Let $G=HS$, and let $V_1,V_2,V_3$ denote the three simple self-dual
$kG$-modules of dimension~154, affording the characters
$154_1,154_2,154_3\in\Irr(G)$ respectively. Let $H$ be a maximal subgroup
of $G$ isomorphic to $M_{22}$. Then $154_1|_H=55_1+99_1$, where $55_1$
is the character afforded by the simple endotrivial $kM_{22}$-module $V_{55}$
of Example~\ref{exmp:m22}, and $99_1$ has defect zero. Thus
$V_1|_H\cong V_{55}\oplus (\text{proj})$ is endotrivial, and so is $V_1$
(see \cite[Lem.~2.2]{LMS}). The characters $154_2,154_3$ lie in the
non-principal block of full defect $B_2$.  Let $e_2$ be the corresponding
block idempotent.
Then $G$ has a maximal subgroup $L\cong\PSU_3(5):2$ with a non-trivial
linear character  $1_2$ such that $e_2\cdot \Ind_{L}^G(1_2)=154_2$.
Reducing modulo~$3$, this proves that  $V_2$ is a trivial source module,
hence endotrivial by Theorem~\ref{prop:torchar}. So is $V_3$, the image of
$V_2$ under the outer automorphism of $G$.
\end{exmp}

\begin{exmp}[The four characters of degree~61776 of $6.Fi_{22}$ in
  characteristic~$\ell=5$]   \label{exmp:6Fi22} \vbox{\  }
Let $G=6.Fi_{22}$. The characters $61776_1,61776_3$ belong to block~$107$
and $61776_2,61776_4$ to block~$108$. Let $e_{107},e_{108}$ denote the
corresponding block idempotents respectively.
Then $G$ has two  non-conjugate subgroups $H$ and $L$ isomorphic to
$\OO_8^+(2).3.2$, such that $e_{107}\cdot\Ind_{H}^G(1_{H})=61776_1$,
$e_{108}\cdot\Ind_{H}^G(1_{H})=61776_2$,
$e_{107}\cdot \Ind_{L}^G(1_{L})=61776_3$ and
${e_{108}\cdot\Ind_{L}^G(1_{L})=61776_4}$.
Therefore the simple reductions $V_{61776_1}$, $V_{61776_2}$, $V_{61776_3}$,
$V_{61776_4}$ modulo $5$ of these characters are trivial source
modules and  are endotrivial by Theorem~\ref{prop:torchar}.  We recall from
\cite[Rem.~7.2]{LMS} that $Fi_{22}$ also has a simple trivial source endotrivial
module $V_{1001}$ affording $1001_1\in\Irr(Fi_{22})$, and $3.Fi_{22}$ has
two faithful simple trivial source endotrivial modules $V_{351}$, $V_{351}^{*}$
affording $351_1,351_2\in\Irr(3.Fi_{22})$.
\end{exmp}

The structure of the group $T(G)$ is not known for the sporadic groups and
their covers. The above examples enable us to describe explicitly the
elements of the torsion subgroup $TT(G)$ in these cases.

\begin{cor}   \label{cor:torsexmp}
 \vbox{\  }
 \begin{enumerate}
  \item[\rm(a)] In characteristic $3$, with notation as in
   Examples~\ref{exmp:m22} and~\ref{exmp:2m22} and with $M$ an indecomposable
   endotrivial $kM_{22}$-module affording the character $154_1\in\Irr(M_{22})$
   we have
   $$\begin{aligned}
     TT(M_{22})=\langle[V_{55}],[M]\rangle\cong \ZZ/2\oplus\ZZ/2,\\
     TT(2.M_{22})=\langle[V_{10}],[V_{154}]\rangle\cong \ZZ/2\oplus\ZZ/4.
   \end{aligned}$$
  \item[\rm(b)] In characteristic $3$, with notation as in
   Example~\ref{exmp:HS} we have
   $$TT(HS)=\langle [V_2],[V_3]\rangle \cong \ZZ/2\oplus\ZZ/2.$$
   \item[\rm(c)] In characteristic $5$, with notation as in
   Example~\ref{exmp:6Fi22} we have
   $$\begin{aligned}
    TT(Fi_{22})&=\langle [V_{1001}]\rangle\cong\ZZ/2,\\
    TT(3.Fi_{22})&=\langle[V_{1001}], [V_{351}]\rangle\cong\ZZ/2\oplus\ZZ/3,\\
    TT(6.Fi_{22})&=\langle [V_{1001}],[V_{61776_1}]\rangle\cong\ZZ/2\oplus\ZZ/6.
   \end{aligned}$$
 \end{enumerate}
\end{cor}

\begin{proof}
For $G\in\{M_{22},2.M_{22}, HS, Fi_{22},3.Fi_{22},6.Fi_{22}\}$, let $\ell$ be
as in the statement, $P\in\Syl_\ell(G)$ and set $N:=N_G(P)$. In all cases
the group $TT(G)$ injects via restriction into the group $X(N)$ of
one-dimensional $kN$-modules, so that its elements are isomorphism classes of
endotrivial trivial source modules, which are the Green correspondents of
modules in~$X(N)$. (See~\cite[\S 4]{C12}.)
\par
\rm(a) If $G=2.M_{22}$, then $P\cong 3^2$ and $X(N)\cong 2\times 4$.
The group $2.M_{22}$ has six simple trivial source endotrivial modules: $k$,
the modules $V_{154},V_{154}^{*}$, $V_{10}$, $V_{10}^{*}$ of
Example~\ref{exmp:2m22}, as well as the module $V_{55}$ of
Example~\ref{exmp:m22} (seen as a $k[2.M_{22}]$-module). As $TT(G)$ injects
in $X(N)$, it must have 8 elements, hence $TT(G)\cong\ZZ/2\oplus\ZZ/4$. The
set of generators is obvious since both $V_{10}$ and $V_{154}$ are not
self-dual, therefore of order~4.  \par
The claim that $TT(G)\cong(\ZZ/2)^2$ for $G=M_{22}$ follows from the above,
since $TT(G)$ is a subgroup of $TT(2.M_{22})$ via inflation and $X(N)\cong 2^2$
in this case. Using the $3$-decomposition matrix of $G$ we conclude that
with $154_2,10_1\in\Irr(2.M_{22})$ the tensor product
$V_{154}\otimes_{k}V_{10}$ equals $M\oplus Q$, where $M$ is an
indecomposable endotrivial $kG$-module affording the character
$154_1\in\Irr(G)$ and $Q$ is a projective $kG$-module. Whence the set of
generators.\par
\rm(b) For $G=HS$ we have $P\cong 3^2$, $N\cong 2\times(3^2.SD_{16})$ and
$X(N)\cong 2^3$. If $e_0$ denotes the principal block idempotent of $kG$,
then $e_0\cdot \Ind_{L}^G(1_{L})=175_1$
with $L\cong\PSU_3(5):2$ as in Example~\ref{exmp:HS}. It follows that the
$kG$-Green correspondent of a one-dimensional module in $X(N)$ affords the
character $175_1\in\Irr(G)$ but is not endotrivial by
Theorem~\ref{prop:torchar}. This together with Example~\ref{exmp:HS} forces
$TT(G)=\langle [V_2],[V_3]\rangle \cong (\ZZ/2)^2$.\par
\rm(c)  For $G=Fi_{22}$ we have $X(N)\cong 4$. The module $V_{1001}$ is
endotrivial
and trivial source, thus must be the Green correspondent of a module in $X(N)$,
whereas the Green correspondents of the two other non-trivial modules in $X(N)$
are not endotrivial: this follows easily by inducing the corresponding linear
characters of $N$ to $G$ and checking that their characters cannot satisfy the
criterion of Theorem~\ref{prop:torchar}. Whence
$TT(G)=\langle [V_{1001}]\rangle\cong\ZZ/2$.
\par
For $G=3.Fi_{22}$, $X(N)\cong 3\times 4$ and $TT(Fi_{22})\leq TT(G)$ via
inflation. Thus in view of Example~\ref{exmp:6Fi22}, we must have
$TT(G)=\langle[V_{1001}], [V_{351}]\rangle\cong\ZZ/3\oplus\ZZ/2$. For
$G=6.Fi_{22}$, $X(N)\cong 6\times4$ and $TT(3.Fi_{22})\leq TT(G)$. Now the
four faithful simple modules of dimension 61776 of Example~\ref{exmp:6Fi22}
are not self-dual. This forces
$TT(G)=\langle [V_{61776_1}],[V_{1001}]\rangle\cong\ZZ/6\oplus\ZZ/2$.
\end{proof}

\section{Special linear groups} \label{sec:linear}

In this section we investigate simple endotrivial $k\SL_n(q)$-modules
for $n\ge3$, where $k$ is a field of characteristic~$\ell$ not dividing $q$.
(The case of $n=2$ or that $\ell|q$ was already considered in
\cite[Prop.~3.8 and Thm.~5.2]{LMS}.) Furthermore, we may assume that $\ell\ne2$
by \cite[Thm.~6.7]{LMS}. Also, we exclude the case of cyclic Sylow
$\ell$-subgroups for the moment, partial results for that situation will be
given in Section~\ref{subsec:SLcyclic}.

Our argument will proceed in three steps. First we treat unipotent characters,
then we deal with the case when $\ell$ divides $q-1$, and finally we consider
the general case. But first we need to collect some auxiliary information.

\subsection{Regular elements and maximal tori}   \label{subsec:3.1}
Let $F:\GL_n\rightarrow\GL_n$ be the standard Frobenius map on the linear
algebraic group $\GL_n$ over a field of characteristic~$p$ corresponding to
an $\FF_q$-rational structure.
The conjugacy classes of $F$-stable maximal tori of $\GL_n$ and of $\SL_n$
are parametrized by conjugacy classes of the symmetric group $\fS_n$ (and thus
by partitions of $n$) in such a way that, if $\bT\le\GL_n$ corresponds to
$w\in\fS_n$ with cycle shape $\la=(\la_1\ge\la_2\ge\ldots)$, then
$|\bT^F|=\prod_i(q^{\la_i}-1)$, while
for $\bT\le\SL_n$ we have $(q-1)|\bT^F|=\prod_i(q^{\la_i}-1)$. In both groups
$\bT$ has automizer $N_{\bG^F}(\bT)/\bT^F$ isomorphic to $C_{\fS_n}(w)$ (see
e.g.~\cite[Prop.~25.3]{MT}).
For a prime $\ell>2$ not dividing $q$ we write $d_\ell(q)$ for the
multiplicative order of $q$ modulo $\ell$.

\begin{lem}   \label{lem:regSLn}
 Let $\la\vdash n$ be a partition, and $\bT$ a corresponding $F$-stable maximal
 torus of $\SL_n$. Assume that either all parts of $\la$ are distinct, or
 $q\ge3$ and at most two parts of $\la$ are equal. Then:
 \begin{enumerate}
  \item[\rm(a)] $\bT^F$ contains regular elements.
 \end{enumerate}
 Now let $\ell$ be a prime such that some part of $\la$ is divisible by
 $d:=d_\ell(q)$.
 \begin{enumerate}
  \item[\rm(b)] If either $d>1$, or $\la$ has at least
   two parts then $\bT^F$ contains $\ell$-singular regular elements.
  \item[\rm(c)] If either $d>1$, or $\la$ has at least three parts,
   or $\ell$ divides $(q-1)/\gcd(n,q-1)$ and $\la$ has at least two parts,
   then $\bT^F$ contains $\ell$-singular regular elements with non-central
   $\ell$-part.
 \end{enumerate}
\end{lem}

\begin{proof}
Write $\la=(\la_1\ge\ldots\ge\la_s)$. First consider the corresponding torus
$\tilde\bT=\bT Z(\GL_n)$ in $\GL_n$. It is naturally contained in an
$F$-stable Levi subgroup $\prod_i \GL_{\la_i}$ of $\GL_n$ such that
$\bT_i:=\GL_{\la_i}\cap\tilde\bT$ is a Coxeter torus of $\GL_{\la_i}$. In
particular, $\bT_i^F\cong \FF_{q^{\la_i}}^\times$ is a Singer cycle. Let
$x_i\in\bT_i^F$ be such that its eigenvalues are generators of
$\FF_{q^{\la_i}}^\times$, and $\tilde x=(x_1,\ldots,x_s)\in\tilde\bT^F$. If
all $\la_i$ are distinct, then clearly all eigenvalues of $\tilde x$ are
distinct, so $\tilde x$ is regular. Multiplying $x_1$ by the inverse of
$\det \tilde x$ yields a
regular element $x$ in $\bT^F=\tilde\bT\cap\SL_n(q)$. If $q\ge3$
then there are at least two orbits of generators of $\FF_{q^{\la_i}}^\times$
under the action of $\Gal(\FF_{q^{\la_i}}/\FF_q)$, so we may again arrange for
an element with distinct eigenvalues, proving~(a). Clearly, if $\la_i$ is
divisible by $d$ then $o(x_i)$ is divisible by $\ell$. If $d>1$ or $\la$ has
at least two parts, this also holds for our modified element $x$, so we get~(b).
\par
Since $|Z(\SL_n(q))|=(n,q-1)$, (c) is clear when $d>1$. So now assume that
$d=1$, that is, $\ell|(q-1)$. If $\la$ has at least three parts, then we may
arrange so that the $\ell$-parts of the various $x_i$ are not all equal, even
inside $\SL_n(q)$, so that we obtain an element with non-central $\ell$-part.
The same is possible if $(q-1)/(n,q-1)$ is divisible by $\ell$ and $\la$
has at least two parts.
\end{proof}

We will also need some information on centralizers of semisimple elements
in $\PGL_n=\GL_n/Z(\GL_n)$ containing $F$-stable maximal tori from given
classes. For this, let $\bH$ denote an $F$-stable reductive subgroup of
$\PGL_n$. If $\bH$ contains a maximal torus of type $w$, then the Weyl group
$W_H$ of $\bH$ will have to contain a conjugate of $w$. Now $W_H$ is a parabolic
subgroup of the Weyl group $\fS_n$ of $\PGL_n$. Then $F$ permutes the factors
of this Young subgroup, and thus $W_H^F$ is a product of various symmetric
groups, some of them in an imprimitive action. Thus we conclude the following:

\begin{lem}   \label{lem:genSL}
 Let $\bH\le\PGL_n$ be a reductive subgroup containing $F$-stable maximal tori
 corresponding to cycle shapes $\la_1,\ldots,\la_r$. If no intransitive or
 imprimitive subgroup of $\fS_n$ contains elements of all these cycle shapes,
 then $\bH=\PGL_n$.
\end{lem}

\subsection{Unipotent characters of $\SL_n(q)$}   \label{subsec:3.2}

We first investigate possible endotrivial simple unipotent modules of
$\SL_n(q)$. Recall that these are naturally labelled by partitions
$\la\vdash n$ (see e.g. \cite[\S13]{Ca}), and we write $\chi_\la$ for the
complex unipotent character with label $\la$. We first describe their values
on regular semisimple elements.

\begin{prop}   \label{prop:valSLn}
 Let $G=\SL_n(q)$. Let $t\in G$ be a regular semisimple element, and let
 $w\in W=\fS_n$ be the label of the unique $F$-stable maximal torus
 $\bT\le\SL_n$ containing $t$. Let
 $\la\vdash n$ be a partition, $\chi_\la\in\Irr(G)$ the corresponding
 unipotent character, and $\vhi_\la\in\Irr(W)$ the corresponding
 irreducible character of $W$. Then $\rho_\la(t)=\vhi_\la(w)$.
\end{prop}

\begin{proof}
If $t$ is contained in a unique $F$-stable maximal torus $\bT$ of $\SL_n$, then
the character formula in \cite[Prop.~7.5.3]{Ca} for the Deligne--Lusztig
character $R_{\bT,1}$ simplifies to
$$R_{\bT,1}(t)=\frac{1}{|T|}\sum_{g\in G\atop t\in T^g} 1=|N_G(\bT):T|
  =|C_W(w)|$$
(where $T=\bT^F$), and $R_{\bT',1}(t)=0$ for any $F$-stable maximal torus
$\bT'$ not $G$-conjugate to $\bT$. The unipotent characters in type $A$
coincide with the almost characters (see \cite[\S12.3]{Ca}), so
$$\begin{aligned}
\chi_\la(t)=&\frac{1}{|W|}\sum_{x\in W}\vhi_\la(x) R_{\bT_x,1}(t)
  = \frac{1}{|W|}\sum_{x\in [w]}\vhi_\la(x) R_{\bT,1}(t)\\
  = &\frac{|[w]|}{|W|}\vhi_\la(w) |C_W(w)|=\vhi_\la(w),
\end{aligned}$$
where $\bT_x$ denotes an $F$-stable maximal torus parametrized by $x\in W$,
and $[w]$ is the conjugacy class of $w$ in $W$.
\end{proof}

We will also need two closely related statements in our later treatment of
exceptional groups of Lie type. For this, assume that $\bG$ is connected
reductive with Steinberg endomorphism $F:\bG\rightarrow\bG$, and set
$G:=\bG^F$. We consider $s\in G^*$ semisimple. The almost characters in the
Lusztig series $\cE(G,s)$ are indexed by extensions to $W_s.F$ of $F$-invariant
irreducible characters of the Weyl group $W_s$ of $C_{G^*}(s)$,
and we denote them by $R_\tvhi$, for $\tvhi$ a fixed extension of
$\vhi\in\Irr(W_s)^F$.

\begin{prop}   \label{prop:zero}
 In the above setting, let $t\in G$ be regular semisimple and
 $\vhi\in\Irr(W_s)^F$. Then $R_\tvhi(t)=0$ unless $t$ lies in an $F$-stable
 maximal torus $\bT$ such that $\bT^*\le C_{\bG^*}(s)$ and $\bT^*$ is indexed
 by $wF\in W_sF$ such that $\tvhi(wF)\ne0$.
\end{prop}

\begin{proof}
Since $t$ is assumed to be regular, it is contained in a unique $F$-stable
maximal torus $\bT$ of $\bG$. The character formula in \cite[Prop.~7.5.3]{Ca}
for the Deligne--Lusztig characters then shows that $R_{\bT,\theta}(t)=0$
for all $\theta\in\Irr(\bT^F)$ unless $s\in\bT^*$ up to conjugation. So now
assume that $\bT^*\le C_{\bG^*}(s)$ is indexed by the $F$-conjugacy class of
$w\in W_s$. Then by definition the almost character for $\vhi$ is given by
$$R_\tvhi(t)=\frac{1}{|W_s|}\sum_{x\in W_s}\tvhi(xF) R_{\bT_x,\theta}(t)
  =\frac{1}{|W_s|}\sum_{x\sim w}\tvhi(wF) R_{\bT,\theta}(t)=0$$
unless $\tvhi(wF)\ne0$, as claimed, where the sum runs over the $F$-conjugacy
class of $w$.
\end{proof}

\begin{prop}   \label{prop:divis}
 In the above setting, let $t\in G$ be regular semisimple and assume that the
 orders of $t$ and $s$ are coprime. Then $R_\tvhi(t)$ is divisible by
 $|W(T)/W(T,\theta)|$ for all $\vhi\in\Irr(W_s)^F$, where $W(T)$ is the Weyl
 group of the unique maximal torus $T$ containing $t$, and $W(T,\theta)$ is the
 stabilizer of $\theta\in\Irr(T)$ corresponding to $s$.
\end{prop}

\begin{proof}
Let $\bT$ be the unique $F$-stable maximal torus of $\bG$ containing $x$,
and $T=\bT^F$. Let $(T,\theta)$ correspond to $s\in G^*$. By
\cite[Prop.~7.5.3]{Ca} we have
$$|R_{\bT,\theta}(t)|=\frac{1}{|T|}\sum_{g\in N_G(\bT)} \theta(t^g)=
  \frac{|N_G(\bT)|}{|T|}=|W(T)|,$$
since the order of the linear character $\theta$ is coprime to that of $t$ by
assumption, and so $\theta$ takes value~1 on any conjugate of $t$. Thus
$$R_\tvhi(t)=\frac{1}{|W_s|}\sum_{x\sim w}\tvhi(wF) R_{\bT,\theta}(t)
  =\pm\frac{|W(T)|\,|W_s/W(T,\theta)|}{|W_s|}\tvhi(wF)
  =\pm\frac{|W(T)|}{|W(T,\theta)|}\tvhi(wF)$$
as claimed.
\end{proof}

\begin{lem}   \label{lem:valspecial}
 Let $n=2d+r$ with $d\ge2$ and $1\le r\le d-1$. There exists a semisimple
 element $t\in\SL_n(q)$ of order $(q^d-1)(q^r-1)$ with centralizer
 $\GL_2(q^d)(q^r-1)$ in $\GL_n(q)$ such that the unipotent character $\rho_\la$
 with $\la=(d+r,r+1,1^{d-r-1})$ takes value $\pm q^d$ on $t$.
 \end{lem}

\begin{proof}
Let $f,g\in\FF_q[X]$ be irreducible of degrees~$d$, $r$ respectively with
$f(0)^2g(0)=(-1)^n$, and $t\in \SL_n(q)$ be a semisimple element with
minimal polynomial $fg$ and characteristic polynomial $f^2g$. Then clearly $t$
has centralizer $C\cong\GL_2(q^d)(q^r-1)$ in $\GL_n(q)$, and thus is as in the
statement. Since the unipotent characters of $\SL_n(q)$ are the restrictions
of those of $\GL_n(q)$, we may now argue in the latter group. Now $t$ is only
contained in the two types of maximal tori $T_1,T_2$ of $C$, of orders
$(q^d-1)^2(q^r-1)$ and $(q^{2d}-1)(q^r-1)$, parametrized by the partitions
$\mu_1=(d,d,r)$ and $\mu_2=(2d,r)$ respectively. Thus the Deligne--Lusztig
characters take values
$$R_{T_1,1}(t)=\pm (q^d+1),\qquad R_{T_2,1}(t)=\mp (q^d-1),$$
where the signs are opposed (since $\eps_{T_1}=-\eps_{T_2}$ in the notation
of \cite[Prop.~7.5.3]{Ca}). Furthermore, the Murnaghan--Nakayama formula gives
$\chi_\la(\mu_1)=(-1)^{d+r+1}=-\chi_\la(\mu_2)$ for the values of the
corresponding characters of $\fS_n$. The value $\rho_\la(t)$ can now be
computed as in the proof of Proposition~\ref{prop:valSLn}.
\end{proof}

We will also need a $q$-analogue of Babbage's congruence, which was first
shown by Andrews \cite{An99} under stronger assumptions; here for integers
$0\le k\le n$ we set
$$\left[n\atop k\right]_x:=\prod_{i=1}^k\frac{x^{n-k+i}-1}{x^i-1}\ \in\ZZ[x].$$

\begin{lem}   \label{lem:Andrews}
 Let $x$ be an indeterminate and $d,h\ge2$. Then in $\ZZ[x]$ we have
 $$\left[hd-1\atop d-1\right]_x
   \equiv x^{(h-1)\binom{d}{2}}\pmod{\Phi_d(x)^2}.$$
\end{lem}

\begin{proof}
When $d$ is a prime, then $\Phi_d(x)=(x^d-1)/(x-1)$, and in that case the
claim is proved in \cite[Thm.~2]{An99}. But inspection shows that the argument
given there is valid for all $d\ge2$ if congruences are considered just modulo
$\Phi_d(x)^2$ instead of $(x^d-1)^2/(x-1)^2$.
\end{proof}

\begin{prop}   \label{prop:unipSL}
 Let $\chi$ be a non-trivial unipotent character of $G=\SL_n(q)$, $n\ge3$,
 and $2\ne\ell{\not|}q$ a prime such that $n\ge 2d$ where $d:=d_\ell(q)$.
 Then $\chi$ is not the character of a simple endotrivial module for any
 central factor group of $G$.
\end{prop}

\begin{proof}
Unipotent characters have the centre of $G$ in their kernel, hence we may see
$\chi$ as a character of any central factor group $S$ of $G$.
Let $\la\vdash n$ denote the label for $\chi$. First note that we have
$\la\ne(n)$, since $(n)$ parametrizes the trivial character of $G$. The
Steinberg character, parameterized by $(1^n)$, can never be the character of
a simple endotrivial $kS$-module, for it has degree $q^{\binom{n}{2}}$, so
can be congruent to $\pm1$ modulo $|S|_\ell$ only when $n=\ell=3$,
$q\equiv4,7\pmod9$. But by \cite[Tab.~4]{Ku00} the Steinberg character is
reducible modulo~$3$ in this case.
\par
If $\chi$ is endotrivial, then in particular its values on $\ell$-singular
elements have to be of absolute value one (see \cite[Cor.~2.3]{LMS}). We start
with the case that
$d=1$, so $\ell|(q-1)$ (and hence $q\ge4$ as $\ell\ne2$). First assume that
moreover $(q-1)/(n,q-1)$ is divisible by $\ell$, so $|Z(G)|=(n,q-1)$ is not
divisible by the full $\ell$-part of $q-1$.
Let $T$ be a maximal torus of $G$ corresponding to an element $w\in\fS_n$ of
cycle shape $(n-1)(1)$. By Lemma~\ref{lem:regSLn}, $T$ contains a regular
element $x$ of $G$ with non-central $\ell$-part. Then $\chi(x)=\vhi_\la(w)$ by
Proposition~\ref{prop:valSLn}. By the Murnaghan--Nakayama formula the latter
is zero unless $\la$ has an $n-1$-hook. Similarly, with a maximal torus
corresponding to the cycle shape $(n-2)(2)$ we see that $\la$ has to possess
an $n-2$-hook. When $n\ge6$ we may also argue that $\la$ has an $n-3$-hook,
using the cycle shape $(n-3)(3)$. The only partitions possessing all three
types of hooks are $(n)$ and $(1^n)$, which we excluded above. When
$n\le5$, there are also the possibilities $\la=(2,2), (3,2)$ and $(2^2,1)$.
From the known values of unipotent characters in $\SL_4(q)$ and $\SL_5(q)$
provided in \Chevie\ \cite{Chv} it follows that all three characters vanish
on the product of a Jordan block of size $n-1$ with a commuting non-central
$\ell$-element.
\par
Next assume that $d=1$, but now $(q-1)/(n,q-1)$ is prime to $\ell$, so
that in particular $\ell$ divides $n$. Thus if $n\le6$, we only need to
consider the pairs $(n,\ell)\in\{(3,3),(5,5),(6,3)\}$. If
$\ell$ divides $|Z(S)|$, then $\chi$ cannot take absolute value one values on
central $\ell$-elements. Otherwise, it can be checked from the known character
tables of unipotent characters \cite{Chv} that there are no examples of degree
congruent to $\pm 1$ modulo $|S|_\ell$.
For $n\ge7$, arguing with maximal tori corresponding to cycle shapes
$(n-2)(1)^2$, $(n-3)(2)(1)$, we see by Lemma~\ref{lem:regSLn} that $\la$ must
possess $n-2$- and $n-3$-hooks, whence $\la$ or its conjugate
partition is one of $(n),(n-1,1),(n-3,3),(n-4,2^2)$ (recall that $q\ge4$).
Now removing an $n-3$-hook from the second and third of these partitions or
their conjugate partitions leaves a 2-core, so the unipotent characters
labelled by these vanish on regular semisimple elements of a torus of type
$(n-3)(2)(1)$. Finally, the character labelled by the partition $(n-4,2^2)$
vanishes on regular elements in a torus of type $(n-4)(3)(1)$, so is not
endotrivial. This completes the discussion for $d=1$.
\par
So now let us assume that $d\ge2$.
Write $n=ad+r$ with $0\le r<d$. Let $T$ be a maximal torus of $G$
corresponding to an element $w\in\fS_n$ of cycle shape $(n-r)(r)$. As before
by Lemma~\ref{lem:regSLn} and Proposition~\ref{prop:valSLn} we conclude that
there is an $\ell$-singular regular element $x\in T$ such that
$\chi(x)=\vhi_\la(w)$. By the Murnaghan--Nakayama formula the latter is zero
unless $\la$ has an $n-r$-hook such that the partition obtained by removing
that hook is
an $r$-hook. Now first assume that $r\ge3$. Then similarly, with a torus
corresponding to an element $w\in\fS_n$ of cycle shape $(n-r)(r-1)(1)$ we see
that the remaining partition has to possess an $r-1$-hook. The only partitions
with that property are (up to taking the conjugate partition),
$\la=(n-r-s+1,2^{s},1^{r-s-1})$ for some $0\le s\le r-1$, and
$\la=(n-r-s,r+1,1^{s-1})$ for some $1\le s\le n-2r-1$.
Next, consider tori corresponding to the cycle shape $(n-d)(d)$, which again
contain $\ell$-singular regular elements. This forces $\la$ to possess an
$n-d$-hook such that removing it leaves a $d$-hook. For the first type of
partition, this is only possible when $s=0$, so $\la=(n-r+1,1^{r-1})$ is a
hook. For the second type, it forces $s=d-r$, whence $\la=(n-d,r+1,1^{d-r-1})$.
So, up to taking conjugates, $\la$ is one of $(n-r+1,1^{r-1})$ or
$(n-d,r+1,1^{d-r-1})$.
\par
Next we consider the situation where $r\le2$. First assume that $r=0$. Then
$\la$ has an $n$-hook, so it is a hook. The cycle shapes $(n-d)(d)$ and
$(n-d)(d-1)(1)$ then show that $\la=(n-d+1,1^{d-1})$ or the conjugate
partition.
\par
If $r=1$ then $\la$ has an $n-r=n-1$-hook, so $\la=(n-s,2,1^{s-2})$ for some
$2\le s\le n/2$. The cycle shape $(n-d)(d)$ then forces
$\la=(n-d,2,1^{d-2})$.
\par
If $r=2$ (so $d\ge3)$ then $\la$ has an $n-2$-hook, so $\la=(n-s,3,1^{s-3})$
for some $3\le s\le n-3$ or $\la=(n-s,2^2,1^{s-4})$ for some $4\le s\le n-2$,
or $\la=(n-1,1)$ or $\la=(n-2,2)$. Again, the cycle shape $(n-d)(d)$ rules out
the last case; the first and second case are only possible for $s=d$, so we
are left with $(n-d,3,1^{d-3})$ (with $d\ge3$), $(n-d,2^2,1^{d-4})$
(with $d\ge4$), and $(n-1,1)$. This concludes our discussion of the various
possibilities for $r$.
\par
Now allow $r$ to be arbitrary again. If $n\ge 3d$ and $r>0$, consider the
cycle shape $(n-d-r)(d)(r)$. Since $(n-d,r+1,1^{d-r-1})$ does not have an
$n-d-r$-hook, this case is out. The cycle shape $(n-d-1)(d)(1)$ shows that
$(n-r+1,1^{r-1})$ cannot occur, and neither can $(n-d+1,1^{d-1})$ for
$n\ne 2d$ nor $(n-d,2^2,1^{d-4})$ for $r=2$ and $n>2d+2$ occur. Finally,
the cycle shape $(d+1)(d)(1)$ rules out $(d+2,2^2,1^{d-4})$.
So at this point, up to taking conjugates $\la$ can only be one of
$(d+r,r+1,1^{d-r-1})$ with $d>r>0$, or $(d+1,1^{d-1})$. In particular $n<3d$
and so $G$ has $\ell$-rank at most~2.
\par
We deal with the cases individually.
For $\la=(d+1,1^{d-1})$ the degree of $\chi_\la$  is
$$\chi_\la(1)=q^{\binom{d}{2}}\left[2d-1\atop d-1\right]_q$$
(see e.g. \cite[13.8]{Ca}), so that by Lemma~\ref{lem:Andrews} we have that
$$\chi_\la(1)\equiv q^{d(d-1)}\pmod{\Phi_d(q)^2}.$$
Since $\Phi_d(q)^2$ divides $|G|$, we see that $\chi_\la(1)$ is not
congruent to $\pm1$ modulo $|G|_\ell$. The degree of the unipotent character
labelled by the conjugate partition only differs by a power of $q$, so the
same argument applies.
\par
It follows from \cite[Thm.~4.15]{JM97} that the Specht module of the Hecke
algebra $H$ of type $A_{n-1}$ indexed by the $d$-regular partition
$\la=(d+r,r+1,1^{d-r-1})$ or its conjugate is reducible if $q$ is specialized
to a $d$th root of unity, for $r=1,\ldots,d-2$. Since the decomposition matrix
of $H$ embeds into that of $\SL_n(q)$ (see
e.g.~\cite[Thm.~4.1.14]{GJ}), the corresponding unipotent characters are
reducible modulo~$\ell$, so do not correspond to simple endotrivial modules.
(Alternatively, one could also appeal to Lemma~\ref{lem:valspecial} for the
partition $(d+r,r+1,1^{d-r-1})$, but not for its conjugate.)
It remains to consider $\la=(d+r,r+1,1^{d-r-1})$ and its conjugate for
$r=d-1$, that is, $\la=(2d-1,d)$. Then
$$\chi_\la(1)=q^d\left[3d-1\atop d-1\right]_q\equiv q^{d^2}\pmod{\Phi_d(q)^2}$$
again by Lemma~\ref{lem:Andrews} (resp.~ times some power of $q$ for the
conjugate partition), and we are done as before.
\end{proof}

\subsection{$\SL_n(q)$ with $\ell|(q-1)$}
We next investigate arbitrary irreducible characters of $\SL_n(q)$ when
$\ell|(q-1)$.

\begin{prop}   \label{prop:SLd=1}
 Let $G=\SL_n(q)$ with $n\ge3$ and $2\ne\ell{\not|}q$ a prime with
 $d_\ell(q)=1$. Then no central factor group of $G$ has a non-trivial simple
 endotrivial module in characteristic~$\ell$.
\end{prop}

\begin{proof}
Let $S$ be a central factor group of $G$, and $V$ a simple endotrivial
$kS$-module, where $k$ is of characteristic~$\ell$. Then $V$ lifts to a
$\CC S$-module by \cite[Thm.~1.3]{LMS}. Let $\chi\in\Irr(S)$ denote the
character of such a lift. We may and will consider $\chi$ as an
irreducible character of $G$. Thus, $\chi$ lies in the Lusztig series
$\cE(G,s)$ of a semisimple element $s$ in the dual group $G^*=\PGL_n(q)$.
By \cite[Prop.~6.3]{LMS} then $s$ must lie in a maximal torus containing a
Sylow 1-torus (as $d_\ell(q)=1$). But Sylow 1-tori are maximal tori of $G^*$,
so $s$ is contained in a maximally split torus. Thus, $C_{G^*}^\circ(s)$
is a (1-split) Levi subgroup of $G^*$. For $s=1$ we obtain unipotent
characters, for which there is no example by Proposition~\ref{prop:unipSL},
so from now on let $s\ne1$.   \par
As in the proof for unipotent characters, first assume that $(n,q-1)$ does not
contain the full $\ell$-part of $q-1$. Note that our assumption $\ell\ne2$
forces $q\ge4$. When $n\ge4$ let $T_1,T_2$ denote maximal tori of $G$
corresponding to the cycle shapes $(n-1)(1)$ and $(n-2)(2)$. By
Lemma~\ref{lem:regSLn} both tori do contain regular elements
with non-central $\ell$-part. Thus by \cite[Prop.~6.4]{LMS}, if $V$ is
endotrivial then $C_{G^*}(s)$ contains conjugates of the dual tori
$T_1^*,T_2^*$. But by Lemma~\ref{lem:genSL} there does not exist a proper
1-split Levi subgroup of $\PGL_n(q)$ containing these two types of tori,
whence we are done.
\par
Now let's consider the case $n=3$ which was excluded above. Here the known
character table of $G$ (see \cite{Chv}) shows that at most the characters of
degree $\chi(1)=q(q^2+q+1)$ might satisfy the necessary conditions about
values on $\ell$-singular elements. As $\chi(1)\equiv3\pmod{(q-1)}$, we must
have $\ell=2$ which was excluded.
\par
It remains to consider the case that $(n,q-1)$ is divisible by the full
$\ell$-part of $q-1$. (So, in particular, $\ell|n$ and hence $n\ne4$). Here,
for $n\ge7$ we argue using maximal tori corresponding to cycle shapes
$(n-2)(1)^2,(n-4)(3)(1)$ and $(n-4)(2)^2$, which all contain regular elements
with non-central $\ell$-part, to rule out all possible proper 1-split Levi
subgroups as $C_{G^*}^\circ(s)$.
\par
Finally we deal with the cases $n=3,5,6$. For $n=6$ (and so $\ell=3$) it can
be checked by Lusztig's parametrization or from the explicitly known lists
that all non-unipotent characters either have degree divisible by~$3$,
or are of defect zero for a Zsigmondy prime divisor $r$ of $q^2+1$.
Since a torus of type $(4)(1)^2$ contains elements with non-central
$\ell$-part and of order divisible by $r$, the latter characters vanish on
such elements, hence cannot belong to endotrivial modules. For $n=5$, the
only non-unipotent characters of degree not divisible by $\ell=5$ are the
five characters of degree $\Ph2^2\Ph3\Ph4\Ph5/5$. But these are of
$r$-defect zero for a Zsigmondy prime divisor of $q^3-1$, and since a
maximal torus of type $(3)(1)^2$ contains regular elements with non-central
$\ell$-part which are $r$-singular, these characters do not lead to examples.
Finally, when $n=3$ the only non-unipotent characters of degree not divisible
by $\ell=3$ are the three characters of degree $\Ph2\Ph3/3$, where
$q\equiv1\pmod3$, but these are reducible modulo $3$ by \cite[Tab.~4]{Ku00}.
\end{proof}

\subsection{The general case}

\begin{thm}   \label{thm:SLn}
 Let $S$ be a central factor group of $\SL_n(q)$ with $n\ge3$. Let
 $\ell{\not|}q$ be such that the Sylow $\ell$-subgroups of $S$ are non-cyclic.
 Then $S$ has no non-trivial simple endotrivial module in characteristic~$\ell$.
\end{thm}

\begin{proof}
As in the proof of Proposition~\ref{prop:SLd=1} we may argue with complex
irreducible characters of $G=\SL_n(q)$. Let $\ell$ be a prime divisor of $|G|$
and set $d:=d_\ell(q)$. We may assume that $\ell$ is odd by \cite[Thm.~6.7]{LMS}.
If $d=1$ then the claim is contained in Proposition~\ref{prop:SLd=1}. So now
assume that $d>1$, and write $n=ad+r$ with $0\le r<d$. Since the Sylow
$\ell$-subgroups of $G$ are non-cyclic we have $a\ge2$, so $n\ge 2d$.
Let $\chi\in\Irr(G)$ be the character of an endotrivial simple $kG$-module.
Then $\chi$ lies in some Lusztig series $\cE(G,s)$, where $s\in G^*=\PGL_n(q)$
is semisimple. Now $G$ contains maximal tori $T_1,T_2,T_3$ parametrized by
the partitions $(n-d,d)$, $(n-r,r)$ and $(n-d-1,d,1)$, all of which contain
$\ell$-singular regular semisimple
elements by Lemma~\ref{lem:regSLn}. As $\chi$ is endotrivial, it cannot vanish
on these elements. But then by \cite[Prop.~6.4]{LMS} the centralizer
$C_{G^*}^\circ(s)$ contains maximal tori of these three types. By
Lemma~\ref{lem:genSL} the only reductive subgroup of $G^*$ containing tori of
all three types is $G^*$ itself, so $s=1$ and $\chi$ is a unipotent character.
Since we have $n\ge 2d$, the claim
now follows from Proposition~\ref{prop:unipSL}.
\end{proof}

\subsection{Exceptional covering groups}

We now discuss the characters of exceptional covering groups of classical
groups. 

\begin{prop}   \label{prop:linexc}
 Let $G$ be an exceptional covering group of a simple group of classical
 Lie type and $\ell$ a prime such that the Sylow $\ell$-subgroups of $G$ are
 non-cyclic. If $V$ is a faithful simple endotrivial $kG$-module, then
 $\ell=3$ and $(G,V)$ are as listed in Table~\ref{tab:ExCov}.
\end{prop}

\begin{table}[htbp]
 \caption{Faithful simple endotrivial modules for exceptional
 covering groups of simple groups of classical Lie type}
  \label{tab:ExCov}
\[\begin{array}{|c|c|l|}
\hline
     G& 3\text{-rank}& \dim V \cr
\hline
   2.\PSL_3(4)& 2&        10,10\cr
 4_1.\PSL_3(4)& 2&      8,8,8,8\cr
 4_2.\PSL_3(4)& 2&  28,28,28,28\cr
\hline
\end{array}\]
\end{table}

\begin{proof}
Let $\chi\in\Irr(G)$ be the character afforded by the lift of a faithful
simple endotrivial $kG$-module $V$. Using \cite[Lem.~2.2, Cor.~2.3]{LMS}, the
known ordinary character tables and decomposition matrices (see \cite{GAP})
we obtain the list of candidates for $\chi$ as listed in Table~\ref{tab:ExCov}.
In particular, we always have $\ell=3$, and none of the exceptional covering
groups of $\PSU_4(2)$, $\PSU_4(3)$, $\PSU_6(2)$, $\Sp_6(2),\OO_7(3)$ and
$\OO_8^+(2)$ leads to a candidate character.
\par
If $(G,\chi(1))=(4_1.\PSL_3(4),8)$, then $\chi\otimes \chi^*$ has one trivial
constituent and one constituent of defect zero. Hence reducing modulo $3$
yields $V\otimes_{k} V^*\cong k\oplus \text{(proj)}$. If $G=2.\PSL_3(4)$ and
$\chi(1)=10$, then $G\leq 2.M_{22}$ and $\chi=\psi|_G$ where
$\psi\in\Irr(2.M_{22})$ is a faithful character of degree 10, which is
afforded by a simple endotrivial $k2.M_{22}$-module $M$
by~\cite[Thm.~7.1]{LMS}. Thus
$V=M|_G$ is endotrivial by \cite[Lem.~2.2]{LMS}.
If $(G,\chi(1))=(4_2.\PSL_3(4), 28)$, then $V$ is endotrivial by
Theorem~\ref{prop:torchar} as it is the Green correspondent of a
one-dimensional $kN_G(P)$-module for $P\in\Syl_3(G)$. Indeed, there
are non-trivial linear characters $1_a$ and $1_{b}$ of $N_G(P)$ such that
$$\Ind_{N_G(P)}^G(1_a)=\Ind_{N_G(P)}^G(1_b)=28_3+28_5+64_3+80_1+80_3\,.$$
Because $1_{a}$ and $1_{b}$ are not dual to each other, it follows that their
$kG$-Green correspondents (which are trivial source modules) afford distinct
complex characters. Therefore $\Ind_{N_G(P)}^G(1_a)$ and
$\Ind_{N_G(P)}^G(1_b)$ are both  the characters of decomposable modules.
Hence it follows from Green correspondence and the possible
degrees and values of trivial source modules given by Lemma~\ref{lem:tschar}
that $28_3$, $28_5$ are the characters afforded by the $kG$-Green
correspondents of $1_a$ and $1_b$. Similarly for $28_4$ and $28_6$ as they
are dual to $28_3$ and $28_5$ respectively.
\end{proof}

\subsection{Cyclic blocks}   \label{subsec:SLcyclic}
To conclude the investigation of $\SL_n(q)$ we discuss cyclic blocks.
Endotrivial modules have dimension prime to $\ell$, hence lie in $\ell$-blocks
of full defect. Therefore if a simple endotrivial module of a finite group $G$
lies in a cyclic block, then the Sylow $\ell$-subgroups of $G$ are cyclic.
We let $\lse(G)$ denote the number of isomorphism classes of simple
endotrivial modules and $\lsb(G)$ the number of blocks containing simple
endotrivial modules.

\begin{prop}   \label{prop:SLcyc}
 Let $G=\SL_n(q)$ be quasi-simple with $n\geq2$. Let $\ell{\not|}q$ be a prime
 such that the Sylow $\ell$-subgroups of $G$ are cyclic, and let $d:=d_\ell(q)$.
 Then $\ell>2$ and the number $\lsb(G)$ is as follows.
 \begin{enumerate}
  \item[\rm(a)] If $d=1$, then $n=2$, and $\lsb(G)=2$ if $q$ is odd and
   $\lsb(G)=1$ if $q$ is even.
  \item[\rm(b)] If $d>1$ then
   $$\lsb(G)=\begin{cases}
    	\gcd(q-1,n) & \text{if  } n=d, \\
    	2 &\text{if } n=d+2 \text{ and }q=2,\\
 	q-1 & else.\\
  \end{cases}$$
 \end{enumerate}
\end{prop}

\begin{proof}
First $G$ does not have cyclic Sylow $2$-subgroups (by the Brauer--Suzuki
theorem), hence $\ell>2$.
Second if $\ell|(q-1)$, then $n=2$ and the simple endotrivial $kG$-modules
are classified in \cite[Prop.~3.8]{LMS}. Therefore we may assume $d>1$ and we
have $n<2d$.   \par
Let $P\in\Syl_\ell(G)$.  By \cite[Lem.~3.2]{LMS}, the number $\lsb(G)$ is
$\frac{1}{e}|X(H)|=\frac{1}{e}|H/[H,H]|_{\ell'}$, where $e$ denotes the
inertial index of the principal block of $G$ and
$H=N_G(\langle u\rangle)=N_G(P)$ for $u\in P$ an element of order $\ell$.
As $d>1$, any $\ell$-block of $G$ is covered by a cyclic defect block of
$\GL_n(q)$, hence $e$ is equal to the inertial index of the principal block
of $\GL_n(q)$, and it follows from \cite[(9)]{FS} that $e=d$.
Finally $|X(H)|$ is calculated in \cite[Thm.~1.2]{CMN14} as follows: if $n=d$,
then $|X(H)|=c'd$ with $c'=\gcd(q-1,(q^d-1)/(q-1))$ which equals $\gcd(q-1,d)$,
whereas $|X(H)|=d(q-1)$ if $n=d+f$ with $f>0$, with the exception that
$|X(H)|=2d$ if $n=d+2$ and $q=2$. The claim follows.
\end{proof}

We recall from \cite[\S 3]{LMS} that a block $B$ of $kG$ containing endotrivial
modules has inertial index equal to that of the principal block.
Moreover simple endotrivial $kB$-modules coincide with the non-exceptional
end nodes of its Brauer tree $\sigma(B)$. A description of the Brauer trees
for $\SL_{n}(q)$ with $d_{\ell}(q)>1$ is provided by \cite{Man}.

\begin{cor}   \label{cor:nrSLn}
 Let $G=\SL_n(q)$, $n>2$, with cyclic $P\in\Syl_{\ell}(G)$, and let
 $d=d_\ell(q)$ (and hence $d>1$). Let $B$ be a block of $kG$ containing a
 simple endotrivial module $V$.
 \begin{enumerate}
  \item[\rm(a)]  If $1<d<n$, then $V$ is the restriction of a simple
   endotrivial module of $\GL_{n}(q)$, and $\sigma(B)$ is an open polygon
   with exceptional vertex sitting at one end. Moreover,
   $$\lse(G)= \begin{cases} 2(q-1)& \text{ if $d=\ell-1$ and $|P|=\ell$},\\
               q-1& \text{ else,}
              \end{cases}$$
   unless $q=2$ and $n=d+2$, in which case
   $$\lse(G)= \begin{cases}
               4 &   \text{ if  $d=\ell-1$ and $|P|=\ell$},\\
               2 &  \text{ else.}
              \end{cases}$$
  \item[\rm(b)] If $d=n$ and $Z(G)=\{1\}$, then $B$ is the principal
   block $B_0$ and $V$ is the restriction of a simple endotrivial module of
   $\GL_{n}(q)$. Moreover,
   $$\lse(G)= \begin{cases}
               2 &   \text{ if $d=\ell-1$ and $|P|=\ell$},\\
               1 &  \text{ else.}
              \end{cases}$$
    In particular, if $d\neq \ell-1$, or $|P|\neq \ell$, then $V$ is trivial.
  \item[\rm(c)] If $d=n$ and $Z(G)\neq \{1\}$, then either $B=B_0$ and $V$
   is the restriction of a simple endotrivial module of $\GL_{n}(q)$, or
   $B\neq B_0$ and $\sigma(B)$ is a star with exceptional vertex in the
   middle and $d/e_{\tilde{B}}>1$ equal-length  branches, where $e_{\tilde{B}}$
   is the inertial index of any block $\tilde{B}$ of full defect of $\GL_{n}(q)$ covering $B$.
 \end{enumerate}
\end{cor}

\begin{proof}
Set $\tilde{G}:=\GL_{n}(q)$. First we claim that the kernel $K$ of the
restriction map ${\Res^{\tilde{G}}_G:T(\tilde{G})\lra T(G):[M]\mapsto [M|_G]}$
is $X(\tilde{G})$, the group of one-dimensional $k\tilde{G}$-modules. As $G$ is
a perfect group, certainly $X(\tilde{G})\subseteq K$. Now assume that $M$ is
an indecomposable  endotrivial $k\tilde{G}$-module such that
$M|_G=k\oplus\text{(proj)}$. Then, as $d>1$, $M$ is $G$-projective so that
$M\,|\,k^{\tilde{G}}$ and in turn by the Mackey formula, $M|_G$ is a summand
of
$$k^{\tilde{G}}|_G\cong \bigoplus_{x\in[\tilde{G}/G]}\tw{x}k\cong
  \bigoplus_{x\in[\tilde{G}/G]}k.$$
Comparing both decompositions of $M|_G$ yields $M|_G\cong k$ and so
$\dim_{k}M=1$.\par
Assume that $1<d<n$. In this situation $|T(\tilde{G})|=(q-1)|T(G)|$ by
\cite[Thm.~1.2(d)(ii)]{CMN14}. Therefore, as $K=X(\tilde{G})\cong C_{q-1}$,
the map $\Res^{\tilde{G}}_G$ is surjective. Thus any block $B$ of $kG$
containing a simple endotrivial module $V$ is covered by a block $\tilde{B}$
of $k\tilde{G}$ containing an endotrivial module.
By \cite[Thm.~3.7]{LMS}, simple endotrivial modules lying in $B$ and
$\tilde{B}$ coincide with the non-exceptional end nodes of $\sigma(B)$ and
$\sigma(\tilde{B})$ respectively. By \cite[Thm.~C]{FS}, $\sigma(\tilde{B})$
is a an open polygon with exceptional vertex sitting at one end. So let
$\tilde{V}$ be the simple endotrivial module at the other end of
$\sigma(\tilde{B})$, and let $\tilde\chi\in\Irr(\tilde{B})$ be the
corresponding character. By \cite[Thm.~1]{Man}, $\sigma(B)$ is a star obtained
by unfolding $\sigma(\tilde{B})$ around its exceptional vertex, so that
$\tilde\chi|_G$ is a sum of $d/e_{\tilde{B}}$ irreducible constituents
labelling the end vertices of $\sigma(B)$.
Now $\tilde{V}|_G=V_0\oplus \text{(proj)}$, where on the one hand $V_0$ is
indecomposable endotrivial by \cite[Lem.~2.2]{LMS}, and on the other hand
$V_0$ is simple since $G\triangleleft \tilde{G}$.  This forces
$\tilde{V}|_G=V_0$ and $\sigma(B)$ to be an open polygon with exceptional
vertex at the end (i.e., a star with one branch). The statements on $\lse(G)$
are then straightforward from Proposition~\ref{prop:SLcyc} and
\cite[Thm.~3.7]{LMS}. \par
Now assume $d=n$. The principal block $B_0$ is covered by the principal
block $\tilde{B}_0$ of $k\tilde{G}$. Therefore the above argument applies
again and shows that $\sigma(B_0)$ is an open polygon with exceptional vertex
at one end and any simple endotrivial $kG$-module is the restriction of a
simple endotrivial $k\tilde{G}$-module. %
If $Z(G)=\{1\}$, then by Proposition~\ref{prop:SLcyc} the unique block
containing simple endotrivial modules is $B_0$. Whence $\lse(G)=2$ if the
exceptional multiplicity is one, i.e., $d=\ell-1$ and $|P|=\ell$, and
$\lse(G)=1$ otherwise. %
If $Z(G)\neq\{1\}$, then $|T(\tilde{G})|=2d(q-1)$ by
\cite[Thm.~1.2(d)(i)]{CMN14}, so that
$\lsb(\tilde{G})=\frac{1}{2d}|T(\tilde{G})|=q-1$ by \cite[Lem.~3.2]{LMS}.
As $K=X(\tilde{G})\cong C_{q-1}$, it follows that each block $\tilde{B}$ of
$k\tilde{G}$ containing an endotrivial module contains a one-dimensional
$k\tilde{G}$-module and covers $B_0$. Therefore, if $B$ is a non-principal
block of $kG$, then it is covered by blocks of $k\tilde{G}$ containing no
endotrivial modules. Then  \cite[Thm. 1]{Man} forces $\sigma(B)$ to be a star
with exceptional vertex in the middle and $d/e_{\tilde{B}}>1$ branches, for
if $d/e_{\tilde{B}}$ were~$1$, then again by a similar argument to that given
in the case $1<d<n$, a simple module at the end of $\sigma({\tilde{B}})$ would
be endotrivial.
\end{proof}

\section{Special unitary groups} \label{sec:unitary}

In this section we classify simple endotrivial modules of special unitary
groups. Many arguments can be copied from the case of special linear groups.
As before, we first study unipotent characters.

\subsection{Unipotent characters of $\SU_n(q)$ with $n\ge3$}
Let $F:\GL_n\rightarrow\GL_n$ be the twisted Steinberg endomorphism with
$\GL_n^F=\GU_n(q)$. Again the conjugacy classes of $F$-stable maximal
tori of $\GL_n$ and of $\SL_n$ are parametrized by partitions of $n$ in such
a way that, if $\bT\le\GL_n$ corresponds to the partition $\la$, then
$|\bT^F|=\prod_i(q^{\la_i}-(-1)^{\la_i})$, while for $\bT\le\SL_n$ we
have $|\bT^F|=\prod_i(q^{\la_i}-(-1)^{\la_i})/(q+1)$. In both cases
$\bT$ has automizer isomorphic to $C_{\fS_n}(w)$. We need some information on
regular elements.

\begin{lem}   \label{lem:regSUn}
 Let $\la\vdash n$ be a partition, and $\bT^F$ a corresponding maximal torus
 of $\SU_n(q)$. Assume that either all parts of $\la$ are distinct, or $q\ge3$
 and at most two parts of $\la$ are equal. Then:
 \begin{enumerate}
  \item[\rm(a)] $\bT^F$ contains regular elements.
 \end{enumerate}
 Now let $\ell$ be a prime such that some part of $\la$ is divisible by
 $d:=d_\ell(-q)$.
 \begin{enumerate}
  \item[\rm(b)] If either $d>1$, or $\la$ has at least two parts then $\bT^F$
   contains $\ell$-singular regular elements.
  \item[\rm(c)] If either $d>1$, or $\la$ has at least three parts,
   or $\ell$ divides $(q+1)/\gcd(n,q+1)$ and $\la$ has at least two parts,
   then $\bT^F$ contains $\ell$-singular regular elements with non-central
   $\ell$-part.
 \end{enumerate}
\end{lem}

\begin{proof}
Write $\la=(\la_1\ge\ldots\ge\la_s)$. As in the proof of Lemma~\ref{lem:regSLn}
the corresponding torus $\tilde\bT$ in $\GL_n$ is a direct product of
$F$-stable factors $\bT_i$ contained in an $F$-stable Levi subgroup
$\prod_i \GL_{\la_i}$, such that $\bT_i^F$ is isomorphic to the cyclic
subgroup $N_i$ of order $q^{\la_i}-(-1)^{\la_i}$ of $\FF_{q^{2\la_i}}^\times$.
Choose $x_i\in\bT_i^F$ such that its eigenvalues are generators of $N_i$,
and set $\tilde x=(x_1,\ldots,x_s)\in\tilde\bT^F$. If all $\la_i$ are distinct,
then all
eigenvalues of $\tilde x$ are distinct, so $\tilde x$ is regular. Multiplying
$x_1$ by the inverse of $\det \tilde x$ yields a regular element $x$ in
$\bT^F=\tilde\bT\cap\SU_n(q)$. If $q\ge3$ (or $\la_i\ne2$) then there are at
least two orbits of generators of $N_i$ under
the action of the Galois group, so we may still arrange for an element with
distinct eigenvalues, which proves~(a).   \par
Clearly, if $\la_i$ is divisible by $d$ then $o(x_i)$ is divisible by $\ell$.
If $d>1$ or $\la$ has at least two parts, this also holds for our modified
element $x$, so we get~(b).
\par
Since $|Z(\SU_n(q))|=(n,q+1)$, (c) is clear when $d>1$. So now assume that
$\ell|(q+1)$. If $\la$ has at least three parts, then we may arrange so that
the $\ell$-parts of the various $x_i$ are not all equal, even inside
$\SU_n(q)$, so that we obtain an element with non-central $\ell$-part, and
similarly if $(q+1)/(n,q+1)$ is divisible by $\ell$ and $\la$ has at
least two parts.
\end{proof}

The assertion and proof of Proposition~\ref{prop:valSLn} on values of unipotent
characters on regular semisimple elements remain valid up to sign by replacing
class functions on the Weyl group by $F$-class functions.

\begin{prop}   \label{prop:unipSU}
 Let $\chi$ be a non-trivial unipotent character of $G=\SU_n(q)$, $n\ge3$, and
 $2\ne\ell{\not|}q$ a prime such that $n\ge 2d$ where $d:=d_\ell(-q)$.
 Then $\chi$ is not the character of a simple endotrivial module for any
 central factor group $S$ of $G$.
\end{prop}

\begin{proof}
From the known table of unipotent characters \cite{Chv} we conclude that there
are no non-trivial endotrivial unipotent characters when $n\le6$. So assume
that $n\ge7$.
Let $\la\vdash n$ denote the label for $\chi$. We may assume that $\la\ne(n)$,
since $(n)$ parametrizes the trivial character of $G$. The
Steinberg character, parameterized by $(1^n)$, has degree $q^{\binom{n}{2}}$,
so can be congruent to $\pm1$ modulo $|G|_\ell$ only when
$n=\ell=3$, $q\equiv2,5\pmod9$, but again the Steinberg character is reducible
modulo~$3$ in this situation by \cite[Lemma~4.3]{KK01}.   \par
First assume that $d=1$. We then argue precisely as in the proof of
Proposition~\ref{prop:unipSL} to deduce that $\chi$ vanishes on some
$\ell$-singular element with non-central $\ell$-part, so $\chi$ is not
endotrivial.
\par
Similarly, when $d\ge2$ we write $n=ad+r$ with $0\le r<d$. We again follow the
arguments in the proof of
Proposition~\ref{prop:unipSL}, using Lemma~\ref{lem:regSUn} and the analogue
of Proposition~\ref{prop:valSLn} to conclude that $\la$ must be one of
$(d+r,r+1,1^{d-r-1})$ with $0<r<d$, or $(d+1,1^{d-1})$, up to taking conjugates.
The degree of $\chi_\la$, for $\la$ one of $(d+1,1^{d-1})$ or $(2d-1,d)$,
is obtained from the one of the corresponding character in $\SL_n(q)$ by
replacing $q$ by $-q$ (see \cite[13.8]{Ca}), so the same congruences as in
the proof of
Proposition~\ref{prop:unipSL} show that these two characters cannot be
endotrivial. Next, by the Ennola dual of Lemma~\ref{lem:valspecial} the
unipotent character indexed by $(d+r,r+1,1^{d-r-1})$, with $0<r<d-1$, takes
value $\pm q^d$ on a semisimple element of order a multiple of $q^d-(-1)^d$
with centralizer $A_2((-q)^d)(q^r-(-1)^r)$, so cannot be
endotrivial. Finally, by Lusztig \cite[Thm.~11.2]{Lu92} the unipotent character
indexed by the conjugate partition $(d-r+1,2^r,1^{d-1})$ vanishes on any
element with unipotent part contained in the closure of its unipotent
support of Jordan type $(d+r,r+1,1^{d-r-1})$. Now $\SU_n(q)$ contains elements
of order $(q^d-(-1)^d)/(q+1)$ which centralize a unipotent Jordan block of
size $d+r$, and that is contained in the closure of the class with Jordan
type $(d+r,r+1,1^{d-r-1})$. This deals with the last open case.
\end{proof}

\subsection{$\SU_n(q)$ with $\ell|(q+1)$}
We next consider the case that $d_\ell(-q)=1$. Here, we actually find examples
of simple endotrivial modules.

\begin{prop}   \label{prop:SUd=1}
 Let $G=\SU_n(q)$ with $n\ge3$, $(n,q)\ne(3,2)$, and $\ell$ a prime divisor
 of $|G|$ with $d_\ell(-q)=1$. Let $V$ be a non-trivial simple $kS$-module
 for some central factor group $S$ of $G$ with $k$ of characteristic $\ell$.
 Then $V$ is endotrivial if and only if $S=\PSU_3(q)$, $\ell=3$,
 $q\equiv2,5\pmod9$ and $\dim V=(q-1)(q^2-q+1)/3$.
\end{prop}

\begin{proof}
Let $V$ be a simple endotrivial $kS$-module. We may and will consider $V$ as
a $kG$-module. Let $\chi\in\Irr(G)$ be the complex character of a lift of $V$.
Then $\chi$ lies in the Lusztig series $\cE(G,s)$ of a semisimple element
$s\in G^*=\PGU_n(q)$. By \cite[Prop.~6.3]{LMS} then $s$ must lie in a maximal
torus containing a Sylow 2-torus (as $d_\ell(-q)=1$). But Sylow 2-tori are
maximal tori of $G^*$, so $C_{G^*}^\circ(s)$ is a 2-split Levi subgroup
of $G^*$. The case of unipotent characters where $s=1$ does not provide
examples by Proposition~\ref{prop:unipSU}, so assume that $s\ne1$.
\par
We now mimic the proof of Proposition~\ref{prop:SLd=1}. If $(n,q+1)$ does not
contain the full $\ell$-part of $q+1$, the case that $n>3$ can be ruled out
by consideration of maximal tori of $G$ corresponding to the cycle shapes
$(n-1)(1)$ and $(n-2)(2)$ (which by Lemma~\ref{lem:regSUn} both contain regular
elements with
non-central $\ell$-part). When $n=3$ the known character table of $G$ (see
\cite{Chv}) shows that at most the characters $\chi$ of degree $q(q^2-q+1)$
might satisfy the necessary conditions about values on $\ell$-singular
elements. As $\chi(1)\equiv-3\pmod{(q+1)}$, we must
have $\ell=2$ which was excluded.
\par
It remains to consider the case that $(n,q+1)$ is divisible by the full
$\ell$-part of $q+1$, so $\ell|n$. Here, for $n\ge7$ we argue using maximal
tori corresponding to cycle shapes $(n-2)(1)^2,(n-4)(3)(1)$ and $(n-4)(2)^2$,
which by Lemma~\ref{lem:regSUn} all contain regular elements with non-central
$\ell$-part, to rule out
all proper 2-split Levi subgroups as $C_{G^*}^\circ(s)$. The cases
$n=5,6$ are excluded precisely as in the proof of Proposition~\ref{prop:SLd=1}.
Finally, when $n=3$ the only non-unipotent characters of degree not divisible
by $\ell=3$ are the three cuspidal characters of degree
$(q-1)(q^2-q+1)/3$ when $q\equiv2,5\pmod9$. These are indeed
irreducible modulo~$3$ by \cite[Lemma~4.3]{KK01}. By \cite[Thm.~7.6]{KK01}
for $S=\PSU_3(q)$ they are the characters of the Green correspondents of
$1$-dimensional modules of $N_S(P)$ for $P\in\Syl_\ell(S)$. Hence they are
endotrivial by Theorem~\ref{prop:torchar} as their values on non-trivial
$3$-elements are $1$ (see \cite[Lemma~3.2]{KK01}). They are not endotrivial
for $G=\SU_3(q)$ as the centre $Z(G)\cong C_3$ acts trivially.
\end{proof}

\begin{rem}
For $S=\PSU_3(q)$, $\ell=3$, $q\equiv2,5\pmod9$, the group of torsion
endotrivial modules $TT(S)$ identifies via Green correspondence with a
subgroup of the group $X(N)\cong (\ZZ/2)^2$ of linear characters of $N:=N_S(P)$
for $P\in\Syl_3(S)$. So in fact we have proved that $S$ has three simple
torsion endotrivial modules. Hence $TT(\PSU_3(q))\cong (\ZZ/2)^{2}$,
generated by the simple endotrivial modules $V$ identified in
Proposition~\ref{prop:SUd=1}
\end{rem}

\subsection{The general case}   \label{subsec:SUn}

\begin{thm}   \label{thm:SUn}
 Let $G=\SU_n(q)$ with $n\ge3$, $(n,q)\ne(3,2)$. Let $\ell{\not|}q$ be such
 that the Sylow $\ell$-subgroups of $G$ are non-cyclic. Let $V$ be a
 non-trivial simple $kS$-module for some central factor group $S$ of $G$ with
 $k$ of characteristic $\ell$. Then $V$ is endotrivial if and only if
 $S=\PSU_3(q)$, $\ell=3$, $q\equiv2,5\pmod9$ and $\dim V=(q-1)(q^2-q+1)/3$.
\end{thm}

\begin{proof}
Let $G=\SU_n(q)$ and $\ell$ a prime divisor of $|G|$ with $d:=d_\ell(-q)$.
We may assume that $\ell$ is odd by \cite[Thm.~6.7]{LMS}. If $d=1$ then the
claim is contained in Proposition~\ref{prop:SUd=1}. So now
assume that $d>1$, and write $n=ad+r$ with $0\le r<d$. Since the Sylow
$\ell$-subgroups of $G$ are non-cyclic we have $a\ge2$, so $n\ge 2d$.
Let $\chi\in\Irr(G)$ be the character of a simple endotrivial $kS$-module.
Then $\chi$ lies in some Lusztig series $\cE(G,s)$, where $s\in G^*=\PGU_n(q)$
is semisimple. Now $G$ contains maximal tori $T_1,T_2,T_3$ of types
$(n-d,d)$, $(n-r,r)$ and $(n-d-1,d,1)$, all of which contain regular semisimple
$\ell$-singular elements by Lemma~\ref{lem:regSUn}. As $\chi$ is endotrivial,
it cannot vanish on these elements. But then by \cite[Prop.~6.4]{LMS} the
centralizer $C_{G^*}(s)$ contains maximal tori of these three types. Using again
Lemma~\ref{lem:genSL} we see that  $C_{G^*}(s)=G^*$, so $s=1$ and $\chi$ is a
unipotent character. But there are no non-trivial simple endotrivial unipotent
characters by Proposition~\ref{prop:unipSU}.
\end{proof}

\begin{rem}
The three simple endotrivial cuspidal modules of $\PSU_3(q)$ for $\ell=3$ of
dimension $(q-1)(q^2-q+1)/3$ are in analogy with the two simple endotrivial
cuspidal modules of $\PSL_2(q)$ for $\ell=2$ of dimension $(q-1)/2$ (see
\cite[Prop.~3.8]{LMS}), which also lie in the Lusztig series of a
quasi-isolated $\ell$-element.
\end{rem}

\begin{proof}[Proof of Theorem~\ref{thm:l-rank}]
The assertion was already proved in \cite[Thm.~1.2]{LMS} for alternating
groups, for sporadic groups, for exceptional groups of Lie type and more
generally for all groups of Lie type in their defining characteristic.
The case of classical groups of types $B_n$, $C_n$, $D_n$ and $\tw2D_n$ is
treated in \cite[??]{LM15}. So the only remaining cases are the linear and
unitary groups, and exceptional covering groups. For these, the claim follows
from Theorem~\ref{thm:SLn}, Proposition~\ref{prop:linexc} and
Theorem~\ref{thm:SUn}.
\end{proof}

\subsection{Cyclic blocks}   \label{subsec:SUcyc}
Next we determine the number of cyclic blocks of $\SU_n(q)$ containing simple
endotrivial modules.

\begin{prop}   \label{prop:SUcyc}
 Let $G=\SU_n(q)$ with $n\geq3$, $(n,q)\ne(3,2)$. Let $\ell{\not|}q$
 be a prime such that the Sylow $\ell$-subgroups of $G$ are cyclic and
 let $d:=d_\ell(-q)$.
 Then $\ell>2$, $d>1$ and the number $\lsb(G)$ of $\ell$-blocks containing
 simple endotrivial modules equals
 $$\lsb(G)=\begin{cases}
    \gcd(q+1,n) & \text{if } n=d, \\
              6 & \text{if } n=d+2 \text{ and }q=2,  \\
  	    q+1 & else.\\
  \end{cases}$$
\end{prop}

\begin{proof}
By Brauer--Suzuki $G$ does not have cyclic Sylow $2$-subgroups, hence $\ell>2$.
Secondly, $\ell$ does not divide $q+1$ as $n>2$. Therefore $d>1$. Also $n<2d$.
\par
Let  $P\in\Syl_\ell(G)$. As in the proof of Proposition~\ref{prop:SLcyc}, we
have again, by \cite[Lem.~3.2]{LMS},
$\lsb(G)=\frac{1}{e}|X(H)|=\frac{1}{e}|H/[H,H]|_{\ell'}$,
where $e$ denotes the inertial index of the principal block $B_0$ of $G$ and
$H:=N_G(P)=N_G(\left<u\right>)$ for $u\in P$ an element of order $\ell$. \par
First by \cite[(3B)]{FS} we have that $e=d$. Indeed since $B_0$ is a
unipotent block, its elementary divisor is the polynomial $X-1$ and we get
$e=e_{X-1}=d_{\ell}(-q)$.   \par
Next we compute $|H/[H,H]|_{\ell'}$. First we assume $n=d$. As above, consider
the twisted Steinberg endomorphism $F:\GL_d\rightarrow\GL_d$ with
$\GL_d^F=\GU_d(q)$. Let $\bT$ denote the maximal torus of diagonal matrices.
Consider the  $d$-cycle $w=(1\;2\;\ldots\;d)$ in the Weyl group
$W=N_{\GL_d}(\bT)/{\bT}\cong \fS_d$. Applying \cite[Prop.~25.3]{MT} yields
$$\bT^{wF}=\{\diag(t,t^{-q},t^{(-q)^2}\ldots, t^{(-q)^{d-1}})\in \bT\mid
  t^{q^d-(-1)^d}=1\}\,$$
so that $\bT^{wF}\cong C_{q^d-(-1)^d}$ and
$N_{\GU_d(q)}(\bT)/\bT^{wF}\cong C_{\fS_d}(w)\cong C_{d}$. Therefore we have
$N_{\GU_d(q)}(\bT)=\bT^{wF}\rtimes C_d\cong C_{q^d-(-1)^d}\rtimes C_d$,
where $C_d$ acts via $F$. Now since $P$ is cyclic and is also a Sylow
$\ell$-subgroup of $\GU_d(q)$, we have $N:=N_{\GU_d(q)}(P)=N_{\GU_d(q)}(\bT)$.
Taking the intersection with $G$ yields
$H=N\cap G\cong C_{(q^d-(-1)^d)/(q+1)}\rtimes C_d$ and if $z\in \GU_{n}(q)$
is a generator for $C_{q^d-(-1)^d}$, then
$H/[H,H]\cong \langle z^{q+1}\rangle\!/\!\langle z^{(q+1)(q+1)}\rangle\rtimes C_d$.
Whence $|X(H)|=cd$ with $c=\gcd(q+1,(q^d-(-1)^d)/(q+1))=\gcd(q+1,d)$,
and we obtain $\lsb(G)=c$ in this case.

Now assume $n>d$. Since $P$ is cyclic we may write $n=d+r$ with $1\leq r<d$
and  regard  $P$ as a Sylow $\ell$-subgroup of
$\GU_d(q)\times \{1\}\le \GU_d(q)\times \GU_r(q)$. Then
$$H=N_{\SU_{n}(q)}(P)=\{\left(\begin{smallmatrix}A & 0\\
0 & B\end{smallmatrix}\right)\in G\mid A\in N_{\GU_d(q)}(P), B\in\GU_r(q)\}.$$
Let $\theta:\GU_d(q)\lra \SU_{n}(q)$ denote the injective homomorphism
defined by
$\theta(A):=\left(\begin{smallmatrix}A & 0\\ 0 & B(A)\end{smallmatrix}\right)$
where $B(A)$ is the diagonal matrix $\diag(\det(A)^{-1},1,\ldots,1)\in\GU_r(q)$.
It follows that $H=S\rtimes\theta(N)$ where
$S:=\{1\}\times \SU_r(q)\le \GU_d(q)\times \GU_r(q)$ and $N:=N_{\GU_d(q)}(P)$.
Therefore $H/[H,H]\cong (S/[S,S])_{\theta(N)}\times N/[N,N]$, where the
subscript $\theta(N)$ means taking the cofixed points with respect to
the action of $\theta(N)$. It follows from the case $n=d$ above that
$N/[N,N]\cong \left<z \right>\!/\!\left<z^{q+1} \right>\times C_d\cong  C_{q+1}\times C_d$.
Moreover $(S/[S,S])_{\theta(N)}$ is trivial when $r=1$ or when $S$ is perfect,
in which cases we obtain $\lsb(G)=q+1$.  \par
The only cases with $S$ not perfect are when $(r,q)\in\{(2,2),(2,3),(3,2)\}$.
If $(r,q)=(2,2)$, then
$S/[S,S]\cong C_2$ so that the action of $\theta(N)$ must be trivial. Hence
$H/[H,H]\cong C_2\times C_{q+1}\times C_d$ and $\lsb(G)=6$. Finally if
$(r,q)=(2,3)$, then $S/[S,S]\cong C_3$, and if $(r,q)=(3,2)$, then
$S/[S,S]\cong C_2\times C_2$, but in both cases the cofixed points
$(S/[S,S])_{\theta(N)}$ are trivial. Hence $H/[H,H] \cong C_{q+1}\times C_d$
and $\lsb(G)=q+1$ in these cases as well.
\end{proof}

As no description of the Brauer trees of $\SU_{n}(q)$ similar to that of
\cite{Man} for $\SL_{n}(q)$ is available we do not provide here a statement
analogous to Corollary~\ref{cor:nrSLn}.

\subsection{Zeroes of characters}   \label{subsec:zeroes}

As one application of our previous considerations we obtain the vanishing
result stated in Theorem~\ref{thm:rank3}:

\begin{proof}[Proof of Theorem~\ref{thm:rank3}]
Among exceptional covering groups, the only cases with $\ell$-rank at least~3
for $\ell$ not the defining prime are the covering groups of $\PSU_6(2)$ with
$\ell=3$ for which the claim can be checked from the known character tables.
Note that $\PSU_4(2)$ has two characters of degree~10 which do not vanish on
3-singular elements, but as $\PSU_4(2)\cong\PSp_4(3)$, this is covered by
case~(1). \par
Now let $G$ be a central factor group of $\SL_n(q)$. We may assume that $n\ge3$
since else Sylow $\ell$-subgroups of $G$ are cyclic. Let $\chi\in\Irr(G)$ be
non-trivial. If $\chi$ is not unipotent, the claim follows from the proofs of
Proposition~\ref{prop:SLd=1} and Theorem~\ref{thm:SLn}. For unipotent
characters the proof of Proposition~\ref{prop:unipSL} gives the result, except
for the Steinberg character $\chi_\St$ and for the case that $d=1$,
$(q-1)/(n,q-1)$ is prime to $\ell$ and $n\le6$, so $(n,\ell)\in\{(5,5),(6,3)\}$.
For $\chi_\St$ note that when $G$ has $\ell$-rank at least
three, there exists an $\ell$-element whose centralizer contains a non-trivial
unipotent element, and $\chi_\St$ vanishes on the product (see
\cite[Thm.~6.4.7]{Ca}). For $\SL_5(q)$ with $\ell=5$ dividing $q-1$ exactly
once, all unipotent characters except those labelled by $(3,2)$ and by
$(2^2,1)$ vanish on suitable regular semisimple elements, and the character
labelled by $(2^2,1)$ vanishes on the product of a 5-singular semisimple
element with centralizer $(q+1)(q-1)^2A_1(q)/5$ with a regular unipotent
element in its centralizer (see the \Chevie-table \cite{Chv}). This
just leaves case~(2). For $\SL_6(q)$ all unipotent characters vanish on some
$3$-singular element.
\par
Finally, let $G$ be a central factor group of $\SU_n(q)$ with $n\ge3$,
$(n,q)\ne(3,2)$. Let $\chi\in\Irr(G)$ be non-trivial. If $\chi$ is not
unipotent, the claim follows from the proof of Proposition~\ref{prop:SUd=1}
and Theorem~\ref{thm:SUn}. For unipotent characters by
Proposition~\ref{prop:unipSU} we only need to discuss the Steinberg
character, for which the claim follows as in the case of $\SL_n(q)$, and the
possibility that $\ell|(n,q+1)$ with $(n,\ell)\in\{(5,5),(6,3)\}$. As before,
using \cite{Chv} we see that the only possibility is the one listed in~(3).
\end{proof}

\section{Exceptional type groups}   \label{sec:exc}

In this section we investigate simple endotrivial modules for exceptional
groups of Lie type. We first discuss the unipotent characters.

\begin{prop}   \label{prop:unipExc}
 Let $G=G(q)$ be a quasi-simple exceptional group of Lie type and
 $\ell{\not|}q$ a prime for which the Sylow $\ell$-subgroups of $G$ are
 non-cyclic. If $\chi\in\Irr(G)$ is the character of a non-trivial simple
 unipotent
 endotrivial $kG$-module, then $G=F_4(2)$, $\ell=5$ and $\chi=F_4^{II}[1]$.
\end{prop}

\begin{proof}
The candidates for characters of simple unipotent endotrivial modules of
exceptional groups of Lie type of rank at least~4 were determined in
\cite[Prop.~6.9]{LMS}. They are given below (where $d=d_\ell(q)$ and the
notation for unipotent characters is as in \cite[\S13]{Ca}):

\[\begin{array}{|r|r|r|l||r|r|r|l|}
\hline
 G& d& \ell& \chi&  G& d& \ell& \chi\\
\hline \hline
     F_4(q)&  4&  5&  F_4^{II}[1]&  \tw2E_6(q)&  4&  5&   \tw2E_6[1];\ \phi_{16,5}\\
     E_6(q)&  4&  5&  \phi_{80,7};\ D_4,r&      E_8(q)& 10& 31&  \phi_{28,68}\\
     E_6(q)&  6& 19&  \phi_{6,25}&           &   &   &              \\
\hline
\end{array}\]
\vskip 5pt
\noindent
The characters $\phi_{80,7}$ of $E_6(q)$ and $\phi_{16,5}$ of $\tw2E_6(q)$
are not endotrivial by Corollary~\ref{cor:casesLMS}. The cuspidal unipotent
character $F_4^{II}[1]$ of $F_4(q)$ is simple endotrivial in characteristic
$\ell|\Phi_4$ if and only if $q=2$, so $\ell=5$. Indeed, for the case $q=2$,
firstly the known ordinary and modular character tables \cite{GAP} show that
$F_4^{II}[1]$ remains irreducible modulo~5. Secondly, $F_4^{II}[1]$ is the
character of a trivial source module as $F_4(2)$ has a subgroup
$H\cong \OO_8^{+}(2):\fS_3$ such that
$F_4^{II}[1]=e_0\cdot \Ind_{H}^{F_4(2)}(1_{a})$, where $1_{a}$ is the
non-trivial linear character of $H$ and $e_0$ is the principal block idempotent
of $F_4(2)$, so endotriviality follows from Theorem~\ref{prop:torchar}.
For $q>2$, let $s\in F_4(q)$ be a semisimple element of order $\Phi_4$ with
centralizer of semisimple type $B_2(q)$. Such elements exist for all $q$
(namely, there are $q^2/4$ such classes when $q$ is even, and $(q^2-1)/4$ when
$q$ is odd). By Lusztig's character formula the value of $F_4^{II}[1]$ on such
an element equals $q(q-1)^2/2$, which is larger than~1 when $q>2$.
\par
We next claim that the unipotent character $\phi_{6,25}$ of $G=E_6(q)$ is
reducible modulo primes $\ell$ dividing $\Phi_6$. For this we use that the
decomposition matrix of the corresponding Iwahori--Hecke algebra $H$ of type
$E_6$ embeds into the decomposition matrix of the unipotent blocks of $G$ (see
e.g.~\cite[Thm.~4.1.14]{GJ}). According to the decomposition matrix of $H$ at
$\Phi_6$ in \cite[Tab.~7.13]{GJ}, the character $\phi_{6,25}$ (which is called
$6_p'$ there) occurs in the reduction of the characters $\phi_{20,2}$ and
$\phi_{30,15}$ (denoted $20_p',30_p'$ respectively). Now assume that the
unipotent character $\phi_{6,25}$ remains irreducible modulo~$\ell$. Then its
Brauer character would have to be a constituent of the reduction of the
unipotent character $\phi_{20,2}$. But the latter has smaller degree than
$\phi_{6,25}$, so this is not possible. Exactly the same argument applies to
the unipotent character $\phi_{28,68}$ of $E_8(q)$ modulo $\Phi_{10}$, by
using the decomposition matrix in \cite[Tab.~7.15]{GJ}.
\par
We next consider the cuspidal unipotent character $\tw2E_6[1]$ of
$G=\tw2E_6(q)$. It lies in the 8-member family $\cF$ of unipotent characters
attached to the largest 2-sided cell of the Weyl group of $G$. Let $s$ be a
regular
semisimple element in a maximal torus of $G$ of order $\Phi_1\Ph2\Ph3\Ph4$,
(so) of order divisible by $\Ph4$. (The automizer of such a torus is isomorphic
to the centralizer of the corresponding parametrizing element in the Weyl
group, of type $D_5(a_1)$, hence cyclic of order~12. It is then easy to see
that such regular elements exist for all $q$.) The argument given in the proof
of Proposition~\ref{prop:valSLn} shows that the values on $s$ of the principle
series almost characters in that family $\cF$ are equal (up to sign) to the
values of the corresponding characters of the Weyl group on the class
$D_5(a_1)$. But from the character table of the Weyl group we see that all
of these vanish. Since semisimple classes are uniform,
this implies that all unipotent characters in $\cF$ vanish on $s$. Thus,
$\tw2E_6[1]$ cannot be endotrivial.
\par
Virtually the same reasoning applies to the unipotent character $D_4,r$ of
$E_6(q)$. It also lies in the 8-member family, and by the calculation above
it vanishes on regular semisimple elements of the maximal torus of order
$\Phi_1\Ph2\Ph4\Ph6$. Again, such regular elements of order divisible by
$\Ph4$ exist for all $q$.
\end{proof}

\begin{thm}   \label{thm:exc}
 Let $G=G(q)$ be a quasi-simple exceptional group of Lie type of rank at
 least~4, $\ell{\not|}q$ a prime dividing $|G|$, and $d=d_\ell(q)$. If
 $\chi\in\Irr(G)$ is the character of a non-trivial simple endotrivial
 $kG$-module, then one of the following occurs:
 \begin{enumerate}
  \item[\rm(1)] The Sylow $\ell$-subgroups of $G$ are cyclic;
  \item[\rm(2)] $(G,\ell,\chi)=(F_4(2),5,F_4^{II}[1])$;
  \item[\rm(3)] $G=E_6(q)$ with $\ell=5|(q^2+1)$ and $\chi$ is the semisimple
   character in the Lusztig series of a semisimple element with centralizer
   $\tw2A_3(q)\Ph1\Ph4$;
  \item[\rm(4)] $G=\tw2E_6(q)$ with $\ell=5|(q^2+1)$ and $\chi$ is the
   semisimple character in the Lusztig series of a semisimple element with
   centralizer $A_3(q)\Ph2\Ph4$; or
  \item[\rm(5)] $G=E_7(q)$ with $\ell|(q^2+1)$.
 \end{enumerate}
\end{thm}

\begin{proof}
Let $G$ be quasi-simple of exceptional Lie type and of Lie rank at least~4.
If a Sylow $\ell$-subgroup of $G$ has rank larger than~2, then there are no
examples by \cite[Thm.~6.11]{LMS}.
An easy check on the order formulas (see \cite[Tab.~24.1]{MT}, for example)
shows that the cases with Sylow $\ell$-subgroups of rank~2 are precisely the
following: $d=3,4,6$ for $F_4(q)$, $d=4,6$ for $E_6(q)$, $d=3,4$ for
$\tw2E_6(q)$, $d=4$ for $E_7(q)$, and $d=5,8,10,12$ for $E_8(q)$. By
\cite[Prop.~6.3]{LMS} if $\chi$ is the character of a simple endotrivial
$kG$-module then it must lie in a Lusztig series $\cE(G,s)$ such that
$s\in G^*$ centralizes a Sylow $d$-torus
of $G^*$. If $s=1$, then $\chi$ is unipotent by definition, and this case was
discussed in Proposition~\ref{prop:unipExc} and leads to case~(2). So $s\ne1$.
The character tables for the three groups $F_4(2)$, $E_6(2)$ and $\tw2E_6(2)$
are known and it can be checked directly that no further case apart from the
ones listed in (3) and (4) arises there.
So we now also assume that $q\ne2$ for types $F_4,E_6$ and $\tw2E_6$. We go
through the various possibilities for $(G,d)$ with $s\ne1$.

\begin{table}[htbp]
 \caption{Maximal Sylow tori}
  \label{tab:reg}
\[\begin{array}{|crc|l|c|}
\hline
     G& d& C_{G^*}(S)'& \Phi_e& \text{possible }C_{G^*}(s)\cr
\hline
 F_4(q)&  3&        A_2(q)&  \Ph1\Ph2&     C,\  A_2(q)^2\\
       &  4&        B_2(q)&  \Ph1\Ph2&       C,\  B_4(q)\\
       &  6&    \tw2A_2(q)&  \Ph1\Ph2&  C,\ \tw2A_2(q)^2\\
 E_8(q)&  5&        A_4(q)&      \Ph3&     C,\  A_4(q)^2\\
       &  8&  A_1(q^4),\tw2D_4(q)& \Ph3,\Ph4&     D_8(q)\\
       & 10&    \tw2A_4(q)& \Ph4,\Ph6& C,\  \tw2A_4(q)^2\\
       & 12& \tw3D_4(q),\tw2A_2(q^2)& \Ph3,\Ph4&       -\\
\hline
\end{array}\]
\end{table}

a) First assume that a Sylow $d$-torus $T_d$ of $G^*$ is a maximal torus, hence
in particular self-centralizing in $G$. The cases are listed in
Table~\ref{tab:reg}. As pointed out above, any endotrivial character of $G$
must thus lie in the Lusztig series of an element $s\in T_d$. In the table we
give one or two centralizers $C=C_{G^*}(S)$ of certain $\Phi_d$-tori $S$ of
$G^*$ of
order $\Phi_d$. By \cite[Prop.~6.4]{LMS}, if $\chi\in\cE(G,s)$ is endotrivial,
then $C_{G^*}(s)$ must contain conjugates of all maximal tori of $C$ containing
regular semisimple elements. In particular, $|C_{G^*}(s)|$ must be divisible by
Zsigmondy primes for the cyclotomic polynomials $\Phi_e$ as given in the
fourth column of the table. It is easily seen that the only possible such
centralizers of elements $1\ne s\in T_d$ are the ones listed in the last
column. In particular, there are no cases when $d=12$ for $E_8(q)$.
\par
We consider the remaining cases in turn. For $G=F_4(q)$ and $d=3$, if $s$
has centralizer $A_2(q)\Ph3$ then $\cE(G,s)$ contains three characters, all
of which have degree $\chi(1)\equiv\pm24\pmod{\Ph3}$. Thus they can possibly
be endotrivial only for $\ell\in\{5,23\}$. But neither of these two primes
has $d_\ell(q)=3$. When $s$ has centralizer $A_2(q)^2$ and hence is isolated
of order~3 and $q\equiv1\pmod3$, then the nine characters $\chi$ in $\cE(G,s)$
satisfy $\chi(1)\equiv\pm8\pmod{\Ph3}$, so necessarily $\ell=7$. Now let $t$
be a regular semisimple $7$-singular element in the Sylow $3$-torus $T_3$ of
$G$, of order prime to~3. Let $\theta\in\Irr(T_3)$ be such that $(T_3,\theta)$
lies in the geometric conjugacy class of $s$. Then $s,t$ have coprime order,
and the stabilizer of $\theta$ in $W(T_3)$ has index~2. Thus $\chi(t)$ is
divisible by~2 for any $\chi\in\cE(G,s)$ by Proposition~\ref{prop:divis}, so
$\chi$ cannot be endotrivial.
The arguments in the case that $G=F_4(q)$ with $d=6$ are quite analogous.
Here we need to consider the characters in Lusztig series corresponding to
elements $s$ with centralizer $\tw2A_2(q)^2$, which can be discarded as in the
previous case.
\par
When $G=F_4(q)$ with $d=4$, for $s$ with centralizer $B_2(q)\Ph4$ the four
characters in $\cE(G,s)$ of degree not divisible by $\Ph4$ have degrees
$\chi(1)\equiv\pm24\pmod{\Ph4}$, so only $\ell=5$ is possible. Now all
semisimple elements of $F_4(q)$ are real, as the Weyl group contains $-1$,
so the characters in the series $\cE(G,s)$ are self-dual. But only the one with
Jordan correspondent in $\cE(B_2(q),1)$ of degree $q\Phi_2^2/2$ satisfies the
condition of
Corollary~\ref{cor:selfdual}. On the other hand, that character is reducible
modulo $\ell=5$, as can be seen from the decomposition matrix of the Hecke
algebra of type $B_2$ modulo $\Phi_4$. So this does not lead to examples.
For the involution $s$ with centralizer $B_4(q)$, the elements in $\cE(G,s)$
satisfy $\chi(1)\equiv\pm3,\pm6\pmod{\Ph4}$, so again necessarily $\ell=5$. But
those of degree congruent~1 modulo~5 are reducible modulo~$\Ph4$, as can again
be seen from the decomposition matrix of the relevant Hecke algebra.
\par
Now consider $G=E_8(q)$. When $d=5$, for $s$ with centralizer $A_4(q)\Ph5$ all
characters in $\cE(G,s)$ of full $\Ph5$-defect have
$\chi(1)\equiv\pm120\pmod{\Ph5}$, so we must have $\ell=11$. Again, all
characters in
$\cE(G,s)$ are self-dual since semisimple elements in $E_8(q)$ are real,
and the only characters of degree congruent~1 modulo~11 are those with
Jordan correspondent in $\cE(A_4(q),1)$ of degree $q\Ph2\Ph4$ and
$q^6\Ph2\Ph4$. Both of these are reducible modulo $\Ph5$ as can be seen from
the corresponding Hecke algebra. For the isolated element $s$ of order~5 with
centralizer $A_4(q)^2$ the corresponding
characters in $\cE(G,s)$ satisfy $\chi(1)\equiv\pm24\pmod{\Ph5}$, so cannot be
endotrivial for primes $\ell$ with $d_\ell(q)=5$.
When $d=10$ the same line of argument leads to the characters in
Lusztig series $\cE(G,s)$ for $s$ with centralizer $\tw2A_4(q)\Ph{10}$ and
$\ell=11$. Here the two characters with Jordan correspondents labelled by the
partitions $(4,1)$ and $(2,1^3)$ satisfy the degree conditions. But by
Lemma~\ref{lem:valspecial} they both vanish on the product of an $\ell$-element
with a regular semisimple element in the maximal torus of $\tw2A_4(q)$ of
order $q^4-1$.  \par
Finally, if $G=E_8(q)$ and $d=8$, then $s$ is isolated with centralizer
$D_8(q)$. The character degrees of $\chi\in\cE(G,s)$ satisfy
$\chi(1)\equiv\pm3,\pm6\pmod{\Ph8}$, but as $d_\ell(q)=8$ we have
$\ell\ge17$, so certainly $\chi(1){\not\equiv}\pm1\pmod\ell$.

\begin{table}[htbp]
 \caption{Rank~2 cases for regular $d$ in exceptional groups}
  \label{tab:exc1}
\[\begin{array}{|cr|c|l|c|l|}
\hline
   G& d& |T|& |T_1|& e& \text{possible }C_{G^*}(s)\cr
\hline
 E_6(q)&     4& \Ph1^2\Ph4^2& \Ph1\Ph2\Ph4\Ph6& 6& \tw2A_3(q)\Ph1\Ph4,\ D_5(q)\Ph1\\
    &     6&   \Ph3\Ph6^2& \Ph1\Ph2\Ph4\Ph6& 4& \tw2A_2(q)A_2(q^2)\\
 \tw2E_6(q)& 4& \Ph2^2\Ph4^2& \Ph1\Ph2\Ph3\Ph4& 3& A_3(q)\Ph2\Ph4,\ \tw2D_5(q)\Ph2\\
        & 3&   \Ph3^2\Ph6& \Ph1\Ph2\Ph3\Ph4& 4& A_2(q)A_2(q^2)\\
\hline
\end{array}\]
\end{table}

b) Now consider the pairs $(G,d)$ collected in Table~\ref{tab:exc1}. In all
these cases, $d$ is a Springer regular number for the Weyl group of $G$. In
particular, the centralizer of a Sylow $d$-torus $S_d$ of $G^*$ is a maximal
torus $T$ of $G^*$, whose order is indicated in the table. Furthermore there
exists a maximal torus $T_1$ in $G$ containing a regular semisimple
$\ell$-singular element $x$ of $G$, of order divisible by $\Phi_e$, with
$e$ as in the 5th column. Thus, if $\chi$ is endotrivial, it cannot vanish on
$x$. But then by \cite[Prop.~6.4]{LMS}, $\chi$ must lie in a Lusztig series
$\cE(G,s)$ with $T_1^*\le C_{G^*}(s)$. As pointed out above, $s$ must also lie
in the centralizer $T$ of the Sylow $d$-torus $S_d$ of $G^*$.
It is now easily seen using \cite[Tab.~3]{B05} that the only centralizers of
elements $1\ne s\in T$ with that property are as given in the 6th column of
Table~\ref{tab:exc1}. We consider these possibilities in turn.   \par
First assume that $G=E_6(q)$ with $d=6$. Then all characters $\chi\in\cE(G,s)$
with $C_{G^*}(s)=\tw2A_2(q)A_2(q^2)$ have $\chi(1)\equiv\pm6\pmod{\Ph6}$,
so necessarily $\ell=7$. As $s$ has order~3, we may use
Proposition~\ref{prop:divis} with $t$ regular of order $\Ph4\Ph6/3$ to
conclude that these $\chi$ are not endotrivial. The same reasoning applies when
$G=\tw2E_6(q)$ and $d=3$. For $G=E_6(q)$ and $d=4$, all characters in
$\cE(G,s)$ with $C_{G^*}(s)=D_5(q)\Ph1$ have degree congruent to $\pm3$ or
$\pm6$ modulo $\Ph4$, so here $\ell=5$. But the order of $s$ divides $\Ph1$,
and $5{\not|}\Ph1$ as $d_\ell(q)=4$, so we may apply again
Proposition~\ref{prop:divis} with $t$ regular of order
$\Ph4\Ph6/\gcd(\Ph1,\Ph4\Ph6)$ to rule out this case. The
same applies to $G=\tw2E_6(q)$ and $C_{G^*}(s)=\tw2D_5(q)\Ph2$. When $G=E_6(q)$
with $d=4$ and $C_{G^*}(s)=\tw2A_3(q)\Ph1\Ph4$ then all characters in
$\cE(G,s)$ have degree congruent to $0$ or $\pm24$ modulo $\Ph4$, so once again
$\ell=5$. There are four characters in that Lusztig series of degree prime
to $\ell$, and the two of larger degree are reducible modulo~$\ell$, as follows
from the decomposition matrix of the Hecke algebra of type $\tw2A_3$ modulo
$\Phi_4$. The character of second smallest degree vanishes on a (necessarily
regular) element of order $(q^3+1)\Ph4$, so only the semisimple character in
such a series remains, yielding case~(3). Similarly, for $G=\tw2E_6(q)$ we
need to discuss the Lusztig series for semisimple elements with centralizer
$A_3(q)\Ph2\Ph4$. Here, the decomposition matrix of the Hecke algebra of type
$A_3$ shows that only the semisimple character is irreducible modulo~$\ell$,
as in case~(4).
\par
c) The only remaining case is when $d=4$ for $E_7(q)$, as in case~(5) of the
conclusion.
\end{proof}

\begin{rem}
We note that the irreducible characters $\chi$ of degree $1543879701$ of
$E_6(2)$ and $\psi$ of degree $707107401$ of $\tw2E_6(2)$, respectively,
corresponding to cases~(3) and~(4) of Theorem~\ref{thm:exc} are
not the characters of simple endotrivial modules in characteristic~5.\par
Indeed, let $G$ be one of $E_6(2)$ or $\tw2E_6(2)$. Both characters are
self-dual, therefore if the corresponding modules were simple endotrivial,
then their class in the group $T(G)$ of endotrivial modules would be a torsion
element of order~2 (see Corollary~\ref{cor:selfdual} and its proof). Moreover
such a module has to be the Green correspondent of a $1$-dimensional
$kN_{G}(P)$-module of order~2, where $P\in\Syl_{5}(G)$. In both cases the
normalizer of a Sylow $5$-subgroup has shape $5^{2}:4\fS_{4}$ so that
$X(N_{G}(P))\cong C_{4}$ has a unique element $1_a$ of order 2. Now $G=E_{6}(2)$
has a maximal subgroup $N_{G}(P)\leq H\cong F_{4}(2)$, so that $\chi$ would
have to be the character of the $kG$-Green correspondent of the simple
$kF_{4}(2)$-module affording the character $F_4^{II}[1]$ of
Proposition~\ref{prop:unipExc} (which is itself the $kH$-Green correspondent
of $1_a$). But it can be computed that $\chi$ does not occur as a constituent
of the induction of $F_4^{II}[1]$ to $G$. Hence a contradiction.\par
A similar argument holds for $\psi$ using a maximal subgroup of $\tw2E_6(2)$
isomorphic to $Fi_{22}$ and the self-dual simple endotrivial module of
dimension~1001 of the latter group in Corollary~\ref{cor:torsexmp}(c).
\end{rem}

We can exclude some further instances of Theorem~\ref{thm:exc}(3) and~(4) as
not belonging to simple endotrivial modules:

\begin{lem}
 Let $q$ be a prime power with $5|(q^2+1)$.
 \begin{enumerate}
  \item[\rm(a)] Let $\chi$ be the semisimple character of $E_6(q)$ in a
   Lusztig series parametrized by a 5-element with centralizer
   $\tw2A_3(q)\Ph1\Ph4$. Then $\chi$ is reducible modulo~5.
  \item[\rm(b)] Assume $\gcd(q,6)=1$ and let $\chi$ be the semisimple character
   of $\tw2E_6(q)$ in a Lusztig series parametrized by a 5-element with
   centralizer $A_3(q)\Ph2\Ph4$. Then $\chi$ is reducible modulo~5.
 \end{enumerate}
\end{lem}

\begin{proof}
First consider $G=E_6(q)$. Under our assumption, $\chi$ lies in a unipotent
5-block of $G$. Since $\chi$ is a semisimple character, it is an explicitly
known linear combination of Deligne--Lusztig characters. From this one can
compute that its restriction to $5'$-elements coincides with the restriction
of the following linear combination of unipotent character in the principal
block:
$$-\phi_{1,0}-\phi_{6,1}+\phi_{15,4}+\phi_{D_4:3}+\phi_{81,6}-\phi_{80,7}
  -\phi_{90,8}+\phi_{D_4:21}+\phi_{81,10}.$$
The decomposition matrix of the principal 5-block of $G$ was computed in
\cite[Table~12]{DM14}. In fact, the cited result gives the correct entries
when $25|(q^2+1)$, and lower bounds in case $5||(q^2+1)$. From this, one sees
that the above virtual character is a positive sum of irreducible Brauer
characters
$$\psi_{D_4:3}+\psi_{D_4:21}+\psi_{81,10},$$
where we have labelled the Brauer characters by the corresponding ordinary
unipotent characters using the triangular shape of the decomposition matrix.
So $\chi$ is not irreducible modulo~5.
\par
The previous arguments also apply to $G=\tw2E_6(q)$. Here, the restriction of
$\chi$ to $5'$-classes agrees with the one of
$$-\phi_{1,0}+\phi_{2,4}'-\phi_{1,12}'-\phi_{8,3}'+\phi_{9,6}'+\tw2E_6[1]
  +\phi_{6,6}'-\phi_{16,5}+\phi_{9,6}''.$$
By \cite[Table~26]{DM14}, for $\gcd(q,6)=1$ this is the following positive
linear combination of Brauer characters
$$(1+c_1+c_4)\tw2E_6[1]+\psi_{6,6}'+\psi_{9,6}'',$$
with integers $c_1,c_4\ge0$ and with the same labelling convention as above.
We conclude as before.
\end{proof}

On the other hand, observe that if $\chi$ lies in a Lusztig series as in the
Lemma, but corresponding to a $5'$-element, then it will remain irreducible
modulo~5.

\subsection{Exceptional covering groups}
We conclude our investigations by the consideration of faithful modules for
exceptional covering groups of exceptional groups of Lie type.

\begin{prop}   \label{prop:excmult}
 Let $G$ be one of the exceptional covering groups $2.F_4(2)$, $2.\tw2E_6(2)$
 or $6.\tw2E_6(2)$, and $V$ be a
 faithful simple endotrivial $kG$-module. Then $V$ occurs in
 Table~\ref{tab:exc1c} (with elementary abelian Sylow $\ell$-subgroups of
 order~$\ell^2$), or in Table~\ref{tab:exc1d} (with cyclic Sylow
 $\ell$-subgroups).
\end{prop}

\begin{proof}
In the case of cyclic Sylow $\ell$-subgroups, we may conclude by using the
criteria in \cite[Thm.~3.7]{LMS} and information on the Brauer trees
(kindly provided by Frank L\"ubeck in the case of $\ell=13$ when
$G=6.\tw2E_6(q)$).
\par
The Sylow $\ell$-subgroups of the groups in question are non-cyclic only for
$\ell\le7$. The ordinary character tables are known for all of these groups and
the usual criteria only leave very few cases for $2.F_4(2)$ and $2.\tw2E_6(2)$.
Excluding those characters of $2.F_4(2)$ which are reducible modulo~5 by
using the known 5-modular character table (see \cite{GAP}), and
Corollary~\ref{cor:selfdual} to exclude the $1\,521\,172\,224$-dimensional
module for $2.\tw2E_6(2)$, only the four characters listed in
Table~\ref{tab:exc1c} remain, where for $2.\tw2E_6(2)$ it is not known
whether this character remains irreducible modulo~5.\par
The $22\,100$-dimensional $\FF_7G$-module $V$ for $G=2.F_4(2)$ affording
the character $22100_1$ is endotrivial by Theorem~\ref{prop:torchar}, since
$V$ is the Green correspondent of a $1$-dimensional module for $N_G(P)$ with
$P\in\Syl_{7}(G)$, hence a trivial source module. Indeed, let $e$ be the block
idempotent corresponding to the $7$-block of $G$ containing $22100_1$, then
there is a linear character $1_a\in\Irr(H)$, where
$H\cong 2\times (\tw3D_4(2):3)$ is a maximal subgroup of $G$ containing
$N_G(P)$, such that $e\cdot \Ind_{H}^G(1_a)=22\,100_1$.\par
The 5-modular reductions of the two characters
$12376_1,12376_2\in\Irr(2.F_4(2))$ are also trivial source modules, hence
endotrivial by Theorem~\ref{prop:torchar}. Indeed, the group $G=2.F_4(2)$ has
two non-conjugate subgroups $H_1$, $H_2$ isomorphic to $\OO_8^+(2).\fS_3$,
and if $e'$ denotes the block idempotent corresponding to the
faithful 5-block of $G$ containing both $12376_1$ and $12376_2$, then we have
$e'\cdot\Ind_{H_1}^G(1_{H_1})=12376_1$ and
$e'\cdot\Ind_{H_2}^G(1_{H_2})=12376_2$.\par
(As in \S 2 and \S 3 ordinary irreducible characters are denoted by their
degrees, and the labelling of characters and blocks is that of the $\GAP$
character table libraries \cite{GAP}.)
\end{proof}

\begin{table}[htbp]
 \caption{Candidates in exceptional covering groups of exceptional groups}
  \label{tab:exc1c}
\[\begin{array}{|ccl||ccl|}
\hline
 G& P& \dim(V)&   G& P& \dim(V)\\
\hline \hline
     2.F_4(2)&  5^2&  12376, 12376&   2.\tw2E_6(2)&  5^2&  17736576 ^* \\
             &  7^2&  22100&  & & \\
\hline
\end{array}\]
$(^*)$ this character is possibly reducible, or not endotrivial
\end{table}

\begin{table}[htbp]
 \caption{Simple endotrivial modules for exceptional covering groups of
   exceptional groups with cyclic Sylow subgroups}
  \label{tab:exc1d}
\[\begin{array}{|r|r|r|c|r|}
\hline
 G& \ell& |X(H)|& |X(H)|/e& \dim V\\
\hline \hline
 2.F_4(2)&     13& 24&  2&  2380,2380\\
 2.F_4(2)&     17& 16&  2&  52,1146600\\
\hline \hline
 2.\tw2E_6(2)& 11& 20&  4&  2432\\
 2.\tw2E_6(2)& 11&   &   &  537472\\
 6.\tw2E_6(2)& 11& 60& 12&  90419328,11145019392\ (2\times)\\
 6.\tw2E_6(2)& 11&   &   &  2606204160,5877256320\ (2\times)\\   \cline{2-5}
 2.\tw2E_6(2)& 13& 24&  2&  2432,45696\\
 6.\tw2E_6(2)& 13& 72&  6&  22619520, 6962288256\ (2\times)\\   \cline{2-5}
 2.\tw2E_6(2)& 17& 16&  2&  2432,1521172224\\
 6.\tw2E_6(2)& 17& 48&  6&  494208,1521172224\ (2\times)\\   \cline{2-5}
 2.\tw2E_6(2)& 19& 18&  2&  45696, 22583328768\\
 6.\tw2E_6(2)& 19& 54&  6&  494208, 33949238400\ (2\times)\\
\hline
\end{array}\]
\end{table}

\section{On the Loewy length of principal blocks}   \label{sec:QKKS}

In \cite{KKS}, Koshitani, K\"ulshammer and Sambale investigate principal
$\ell$-block $B_0$ of finite groups of Loewy length~4.
They show that a necessary condition to have Loewy length $LL(B_0)=4$ is the
existence of a character $\chi\in\Irr(G)$ such that $\chi(x)=-1$ for all
$\ell$-singular elements $x\in G$ and $\chi(1)\equiv -1\pmod{|G|_{\ell}}$.
More precisely, in this situation the projective cover of the trivial module
affords the character $1_{G}+\chi$, and the Heller translate $\Omega(k)$,
which is an endotrivial module, affords $\chi$ and has composition length~2.
In particular by Brauer reciprocity the column of the decomposition matrix
corresponding to the trivial Brauer character has exactly two non-zero entries,
so the first Cartan invariant $c_{11}$ of the principal block equals~2.
See \cite[Prop.~4.6, Cor.~4.7]{KKS}. \par
Furthermore, if $\ell\geq 5$ a reduction to simple groups is proven. This 
follows from \cite[Prop.~4.10]{KKS} together with \cite[Prop.~2.10]{KKS}.
It is also shown that Theorem~\ref{thm:LL} holds in characteristic $\ell=2$ 
\cite[Thm.~4.5]{KKS}, as well as in odd characteristic for the alternating
groups \cite[Thm.~3.10 together with Thm.~2.10]{KKS}, the sporadic groups 
\cite[Thm.~4.11]{KKS}, and  for groups of Lie type in their defining
characteristic \cite[Thm.~4.12]{KKS}.
Concerning groups of Lie type in cross characteristic our previous results 
show the following.

\begin{cor}   \label{cor:KKS}
 Let $G$ be one of the simple groups $\PSL_{n}(q)$ or $\PSU_{n}(q)$ with
 $n\geq 3$, and $2\not=\ell{\not|}q$ be such that the Sylow $\ell$-subgroups
 of $G$ are non-cyclic.
 Then the principal $\ell$-block $B_0$ of $G$ does not have Loewy length~4.
\end{cor}

\begin{proof}
If $G=\PSL_n(q)$ (resp.~$G=\PSU_n(q)$) with $n>3$ then it follows from the
proofs of Proposition~\ref{prop:unipSL}, Proposition~\ref{prop:SLd=1}
and Theorem~\ref{thm:SLn} (resp.~Proposition~\ref{prop:unipSU},
Proposition~\ref{prop:SUd=1} and Theorem~\ref{thm:SUn}) that for every
character $\chi\in\Irr(G)$ either $\chi(1)\not\equiv -1\pmod{|G|_{\ell}}$ or
$\chi$ vanishes on some $\ell$-singular element of $G$. The same holds if
$G=\PSL_3(q)$ unless $\ell=3$, $q\equiv 4,7\pmod{9}$ and $\chi$ is the
Steinberg character or one of the three characters of degree
$\Phi_2\Phi_3/3$. But then by \cite[Table~4]{Ku00} the reduction modulo~3
of $\chi$ is not a sum of two irreducible Brauer characters. Similarly for
$G=\PSU_3(q)$ with $q\neq 2$, all characters in $\Irr(B_0)$ are discarded
by the proofs of  Proposition~\ref{prop:unipSU}, Proposition~\ref{prop:SUd=1}
and Theorem~\ref{thm:SUn}, except the Steinberg character when $\ell=3$ and
$q\equiv2,5\pmod9$. But, by \cite[Lemma~4.3]{KK01}, the latter reduces
modulo~$3$ as a sum of 5 irreducible Brauer characters. Therefore in all cases
$LL(B_0)\not=4$ by \cite[Cor.~4.7]{KKS} (note that $\PSU_3(2)$ is solvable).
\end{proof}

To deal with the simple groups of exceptional type we need the following
observation:

\begin{lem}   \label{lem:HC-restr}
 Let $G$ be a finite group with a BN-pair, and $T=B\cap N$. Let $\ell$ be a
 prime dividing $|T|$. Assume that $N$ does not act transitively on the set
 $I$ of non-trivial linear characters of $T$ of order a power of $\ell$.
 Then the 1-PIM of $G$ is a sum of at least three characters.
\end{lem}

\begin{proof}
Assume that $\rho=1_G+\chi$ is projective. Then so is its Harish-Chandra
restriction $\sR_T^G(\rho)$ to $T$. In particular it contains the 1-PIM of $T$.
Now the character of the 1-PIM of $T$ contains the sum over all linear
characters of $T$ of $\ell$-power order. Assume that $N$ has at least three
orbits on this set, with representatives $\psi_1=1_T,\psi_2,\psi_3$ say. Then
the $R_T^G(\psi_i)$ are disjoint, so $\chi$ occurs in exactly one of them,
say for $i=2$. But then by reciprocity it is clear that $\sR_T^G(\rho)$
cannot contain $\psi_3$, a contradiction.
\end{proof}

\begin{prop}   \label{prop:CartanE}
 Let $G$ be a simple group of exceptional Lie type and $\ell>2$ a prime for
 which the Sylow $\ell$-subgroups of $G$ are non-cyclic. Then the principal
 $\ell$-block $B_0$ of $G$ does not have Loewy length~4.
\end{prop}

\begin{proof}
According to \cite[Prop.~4.12]{KKS} we may assume that $\ell$ is not the
defining prime for $G$.
By \cite[Thm.~4.1.14]{GJ} the decomposition matrix of the Hecke algebra of $G$
embeds into the decomposition matrix of $G$. Thus, if the decomposition matrix
of the Hecke algebra has the property that its first column contains at least
three non-zero entries, we must have $c_{11}\ge3$ and so $LL(B_0)\ne4$. \par
The tables in \cite[p.~375--386]{GJ} show that for types $F_4,E_6,\tw2E_6,
E_7$ and $E_8$ this condition is satisfied for the root-of-unity specialization
of the corresponding Hecke algebra $H$ whenever $d=d_\ell(q)>1$ and Sylow
$\ell$-subgroups are non-cyclic. Now the entries in the decomposition matrix
for a finite field specialization of $H$ will be at least as large as for the
root-of-unity specialization, so our claim follows for all these groups as
long as $d>1$.
Similarly, by \cite[p.~373]{GJ} this holds for type $G_2$ when $d\ne1$, type
$\tw3D_4$ when $d\ne1,3$, and for $\tw2F_4(q^2)$ when $\ell{\not|}(q^2-1)$.
\par
To treat $d=1$, so $\ell|(q-1)$, we use Lemma~\ref{lem:HC-restr} with the
natural BN-pair and thus with $T$ the
maximally split torus. It can be checked readily that in all cases, the Weyl
group has at least two orbits on the set of non-trivial $\ell$-elements of
$\Irr(T)$, e.g., by order reasons. For example in $G_2(q)$ we have either
$|\Irr(T)|_\ell-1=8$ (if $\ell=3$), or $|\Irr(T)|_\ell-1\ge15$, and none of
these can be an orbit length for the Weyl group of order $|W|=12$. This
argument also applies to $d=3$ for $\tw3D_4(q)$.
\end{proof}

Taken together with the results of \cite{KKS} and of \cite[Thm.~7.1]{LM15}, this
gives Theorem~\ref{thm:LL}.


\end{document}